\documentclass[12pt]{amsart}
\usepackage[letterpaper,body={6.5in,9in}]{geometry}
\usepackage{amscd}
\usepackage{hyperref}
\usepackage{amsxtra}
\usepackage{amssymb}
\usepackage{enumerate}
\usepackage{amsmath} 
\usepackage{amsthm}
\usepackage{url}
\usepackage[all,ps,cmtip]{xy}

\usepackage{graphicx}
\usepackage[usenames]{color}

 \theoremstyle{plain}
 \newtheorem{thm}{Theorem}[section]
 \newtheorem{cor}[thm]{Corollary}
 \newtheorem{lem}[thm]{Lemma}
 \newtheorem{prop}[thm]{Proposition}

 \theoremstyle{definition}
 \newtheorem{defn}[thm]{Definition}

 \theoremstyle{remark}
 \newtheorem{rem}[thm]{Remark}

\numberwithin{equation}{section}

 \newcommand{\set}[1]{\left\{#1\right\}}

 \newcommand{\p}[1]{\left(#1\right)}

 \newcommand{\CC}{\mathbb{C}}
 \newcommand{\QQ}{\mathbb{Q}}
 \newcommand{\ZZ}{\mathbb{Z}}
 \newcommand{\FF}{\mathbb{F}}
 \newcommand{\PP}{\mathbb{P}}
 \renewcommand{\AA}{\mathbb{A}}

 \DeclareMathOperator{\Gal}{Gal}
 \DeclareMathOperator{\Aut}{Aut}
 \DeclareMathOperator{\SL}{SL}
 \DeclareMathOperator{\PSL}{PSL}
 \DeclareMathOperator{\Sp}{Sp}
 \DeclareMathOperator{\GL}{GL}
 
 \newcommand{\kbar}{\overline{k}}
 \newcommand{\kw}{k(\underline{w})}
 \newcommand{\kf}{k(\underline{f})}
 \newcommand{\id}{\mathrm{id}} 
 \newcommand{\charact}{\mathrm{char}}

 \DeclareMathOperator{\Pic}{Pic}
 
 \DeclareMathOperator{\disc}{disc}
 \DeclareMathOperator{\Jac}{Jac}
 \DeclareMathOperator{\Norm}{Norm}
 \newcommand{\MM}{\mathcal{M}_2}

  \newcommand{\LL}{\mathcal{L}}

\newcounter{nootje}
\newcommand\noot[1]
  {\marginpar{\tiny\begin{minipage}{20mm}\begin{flushleft}\color{red}\thenootje : #1\end{flushleft}\end{minipage}}\addtocounter{nootje}{1}}
\begin{document}
\title{The Arithmetic of Genus Two Curves with (4,4)-Split Jacobians}
\author{Nils Bruin}
\address{Department of Mathematics,
         Simon Fraser University,
         Burnaby, BC,
         Canada V5A 1S6}
\email{nbruin@sfu.ca}
\thanks{Research supported by NSERC}
\author{Kevin Doerksen}
\address{Department of Mathematics,
         Simon Fraser University,
         Burnaby, BC,
         Canada V5A 1S6}
\email{kdoerkse@sfu.ca}
\date{April 21, 2011}

\begin{abstract}
In this paper we study genus $2$ curves whose Jacobians admit a polarized $(4,4)$-isogeny to a product of elliptic curves. We consider base fields of characteristic different from $2$ and $3$, which we do not assume to be algebraically closed.
We obtain a full classification of all principally polarized abelian surfaces that can arise from gluing two elliptic curves along their $4$-torsion and we derive the relation their absolute invariants satisfy.

As an intermediate step, we give a general description of Richelot isogenies between Jacobians of genus $2$ curves, where previously only Richelot isogenies with kernels that are pointwise defined over the base field were considered.

Our main tool is a Galois theoretic characterization of genus $2$ curves admitting multiple Richelot isogenies.

\end{abstract}
\maketitle

\section{Introduction}
Let $k$ be a field and let $C$ be a curve of genus $2$ over $k$.
Let $J=\Jac(C)$ be its Jacobian. The abelian variety is called \emph{decomposable over $k$} if $J$ is isogenous over $k$ to a product of elliptic curves $E_1\times E_2$.

A genus $2$ curve has a decomposable Jacobian if and only if there is a cover $\phi_1\colon C\to E_1$ to an elliptic curve $E'$. If we take $\phi_1$ to be \emph{optimal} (minimal degree would do), this gives rise to a complementary cover $\phi_2\colon C\to E_2$ and an isogeny of a special type
\[\Phi\colon E_1\times E_2\to \Jac(C),\]
which we call an \emph{optimal $(n,n)$-splitting} (Definition~\ref{D:nnsplitting}). The construction is also referred to as \emph{gluing $E_1$ and $E_2$ along their $n$-torsion} and specifying $\Phi$ is equivalent to specifying a Weil-pairing inverting isomorphism $\alpha\colon E_1[n]\to E_2[n]$.

There is a considerable literature on $(n,n)$-splittings, often in the language of elliptic subcovers and mainly dealing with algebraically closed base fields. The first general examples for $n=2$ were given by Legendre and Jacobi (1832). Later Bolza (1887) considered $n=3$ and $n=4$ (see \cite[pp.~477, 480]{krazer}). In recent years, these results have been reconsidered and extended, mainly over an algebraically closed field. For $n=2$ see \cite[Ch.~14]{MR1406090}, for $n=3$ see \cite{MR1363496, MR936803, MR2041085, MR2039100} and for $n=5$ see \cite{MR1363496, degree5}.

In this paper we are concerned with $n=4$. We compare our results to Bolza's \cite{Bolza1887} in Appendix~\ref{S:Bolza}.
Our methods require that the covers and isogenies we consider are separable, so we need that our base field $k$ is not of characteristic $2$. To simplify our computations we sometimes also assume that $\charact(k)\neq 3$ and that $\#k> 5$, but this is not essential for the methods we employ.

One significant advantage of considering optimal $(n,n)$-splittings rather than optimal degree-$n$ covers $\phi\colon C\to E$ is that the codomain of an $(n,n)$-splitting need not be a Jacobian, which means that boundary cases can be treated more uniformly.
Our main result classifies all $(4,4)$-splittings.
\begin{thm}\label{T:44splitclassification}
Let $J$ be a principally polarized abelian surface over a field $k$ with $\charact(k)\nmid 6$ and $\#k > 5$. Then $J$ admits an optimal $(4,4)$-splitting
\[
 \Phi_4\colon E_1\times E_2\to J
\]
if and only if one of the following holds.
\begin{enumerate}
 \item $J=\Jac(C_4)$ where $C_4$ is a genus $2$ curve admitting a model of the form given in Appendix~\ref{S:C4},
 \item $J=\Jac(C_4')$ and $E_2=E_1^{(D)}$, where $D=\disc(E_1)$, with
\[
C_4'\colon Y^2=-64bc\frac{1}{D^3}X^6+\frac{64}{3}b\frac{1}{D^2}X^5+16bc\frac{1}{D^2}X^4+\frac{224}{27}b\frac{1}{D}X^3+
4bc\frac{1}{D}X^2+\frac{4}{3}bX-bc,
\]

 \item $J=E_1\times E_2$ and there is a $3$-isogeny $E_1\to E_2$,
 \item $J=E_1/\langle T_2\rangle\times E_1/\langle T_3\rangle$, where $E_1=E_2$ is an elliptic curve with $E_1[2](k)=\{0,T_1,T_2,T_3\}$
 \item $J=\Re_{k(\sqrt{D})/k}(E_1/\langle T_2\rangle)$, where $D=\disc(E_1)$ is a non-square, $E_1[2](k)=\{0,T_1\}$ and $E_1[2](k(\sqrt{D}))=\{0,T_1,T_2,T_3\}$ and $E_2=E_1^{(D)}$.
\end{enumerate}
\end{thm}
The model $C_4'$ can be obtained as an appropriate specialization of a model for $C_4$.

Combined with the degree $4$ covers explained in Appendix~\ref{S:Bolza}, this also shows that if $\Jac(C_4)$ is $(4,4)$-split then $C_4$ admits an optimal elliptic subcover of degree $4$. See \cite[Cor.~5.19]{kani03} for the result for general $n$.

We use the model (\ref{E:C4}) to describe a birational model of the $2$-dimensional locus of optimally $(4,4)$-split Jacobians in the moduli-space of curves of genus $2$.
The
\emph{Igusa invariants} $I_2$, $I_4$, $I_6$, and $I_{10}$ (see \cite{MR0114819}) of a genus 2 curve $C$ classify the isomorphism class of $C$ over an algebraically closed field.
  They are homogeneous polynomials of degrees 2, 4, 6, and 10 respectively in the coefficients of the defining polynomial for a model of the genus two curve.
This moduli-space is birational to affine $3$-space, as given by the
 \emph{absolute invariants} of a genus two curve \cite{MR0141643}:
\begin{equation}\label{E:AbsInv}
i_1 = 144\frac{I_4}{I_{2}^{2}}, 
	\quad i_2 = -1728\frac{\p{I_2I_4 - 3I_6}}{I_{2}^{3}},
	\quad i_3 = 486\frac{I_{10}}{I_{2}^{5}}.
\end{equation}

\begin{thm}\label{T:main}
The absolute invariants $i_1,i_2,i_3$ of a genus $2$ curve with optimally
 $(4,4)$-split Jacobian satisfy an equation $\LL$, of weighted degree $90$, where $i_1,i_2,i_3$ are given weights $2,3,5$ respectively.
\end{thm}

The equation $\LL$ is too large to reproduce on paper: it consists of 4574 monomials with coefficients having up to 138 digits.  We have therefore made a copy available electronically (see \cite{equation}).  The surface described by $\LL$ is the \emph{Humbert surface} of discriminant 16 (see \cite[Corollary~1.7]{MR1285957}).

\begin{rem}
In Appendix~\ref{S:Bolza} we use Theorem~\ref{T:44splitclassification} to verify a classic result by O. Bolza \cite{Bolza1887}.  We find that one of his equations has a sign error and that our family is birational to his corrected family. 
\end{rem}

Our main tool is the observation that an optimal $(4,4)$-splitting $\Phi_4$ factors as
\[\xymatrix{
E_1\times E_2\ar[r]^-{\Phi_2}\ar@/_2ex/[rr]_-{\Phi_4} &A\ar[r]^\Psi& J,
}\]
where $\Phi_2$ is a $(2,2)$-splitting and $\Psi$ is a polarized $(2,2)$-isogeny. A description of $(2,2)$-split principally polarized abelian varieties is already available and we classify when they admit a further polarized $(2,2)$-isogeny of the desired type.

In general, we have that both $A$ and $J$ are Jacobians. Polarized $(2,2)$-isogenies between Jacobians of genus $2$ curves are known as \emph{Richelot isogenies}.

\begin{rem}
We give a full arithmetic description of Richelot isogenies in Proposition~\ref{T:twist}. Previous literature only considered the case where the kernel is pointwise defined over the base field (see \cite{MR1406090,MR1913484,SmithThesis}).
\end{rem}

The paper is laid out in the following way. In Section~\ref{sec:nnsplit} we give relevant definitions and background material on $(n,n)$-splittings of principally polarized abelian surfaces. In Section~\ref{S:22split} we review the basic description of $(2,2)$-splittings. Section~\ref{S:22maps} collects useful results on Richelot isogenies.

In Section~\ref{S:4x4to2x2} we relate $(4,4)$-splittings to principally polarized abelian surface $A$ admitting multiple polarized $(2,2)$-isogenies.
For a polarized $(2,2)$-isogeny $\Phi: A\to B$ we write $\Phi^*:B\to A$ for the $(2,2)$-isogeny such that $\Phi^*\circ\Phi$ is multiplication by $2$.
Section~\ref{S:2Torsion} considers the case $A=\Jac(C_2)$ and relates the isogenies to Galois-theoretic properties of the Weierstrass points of $C_2$.

\begin{thm}\label{T:galchar}
Let $k$ be a field of characteristic distinct from $2$.
The Jacobian of a genus $2$ curve
$$C: Y^2=f(X)$$
has two $(2,2)$ isogenies over $k$ if and only if the Galois group of $f(X)$ is contained in $C_2\times V_4\subset S_6$ or $\tilde{S}_3=\langle (1,3,5)(2,4,6),(12)(36)(45)\rangle\subset S_6$. In the first case, $\Jac(C)$ has two isogenies $\Phi,\Psi$ such that $\Phi\circ\Psi^*$ is a $(4,2,2)$-isogeny. In the second case, it is a $(4,4)$-isogeny.
\end{thm}

In Section~\ref{S:S3} we apply the results from Section~\ref{S:2Torsion} to derive a model for $C_2$. As a corollary, we obtain a model for the universal elliptic curve over $X_E^{-}(4)$, the modular curve of elliptic curves with $4$-torsion anti-isometric to $E[4]$ (see Proposition~\ref{P:X4-Es}). Silverberg \cite{MR1638488} already derived such formulas, but the ones we list may be of interest since they are shorter.

In Section~\ref{S:C4derivation} we combine results from Sections~\ref{S:22maps} and \ref{S:S3} to derive the model for $C_4$ when both $A$ and $J$ are Jacobians. A finer analysis yields that $C_4'$ can be obtained from $C_4$ as an appropriate limit and that the cases where both $C_4$ and $C_4'$ fail to provide a model of a genus $2$ curve, correspond to surfaces $J$ that are not Jacobians.

\medskip
\noindent\emph{Acknowledgements.}
We would like to thank Everett Howe for pointing us to Oskar Bolza's 1887 result \cite{Bolza1887} described in Appendix~\ref{S:Bolza}. We are also very thankful for the comments of and correspondence with an anonymous referee.  Particularly his insistence on a precise definition of the notion of $(n,n)$-splitting greatly clarified the exposition and ultimately helped us arrive at a correct and complete result.


\section{Split Jacobians}\label{sec:nnsplit}

This section introduces some terminology and reviews some basic facts. We believe all results here are well known but were unable to locate a single source that stated them in the desired form, so we gather them here for the convenience of the reader.

\begin{defn} Let $A$ be an abelian surface over a field $k$. We say that $A$ is \emph{decomposable} if there exist elliptic curves $E_1,E_2$ over $k$ such that $A$ is isogenous to $E_1\times E_2$ over $k$.
\end{defn}

\begin{lem}\label{lem:fincov} Let $C$ be a curve of genus $2$ over a field $k$. If $\Jac(C)$ is decomposable then $C$ admits a finite cover $\phi_1: C\to E_1$ over $k$, where $E_1$ is an elliptic curve.
\end{lem}

\begin{proof}
We write $J=\Jac(C)$. A $k$-rational divisor class of degree $1$ gives rise to an Abel-Jacobi map $C\hookrightarrow J$ over $k$, which allows us to consider $C$ as a subvariety of $J$.
In general, we can use the $k$-rational canonical class $\kappa$ to define a morphism $C\to J$ which over an algebraic closure $\kbar$ corresponds to $\gamma\colon C(\kbar)\to \Pic^0(C/\kbar)$, defined by $P\mapsto [2P]-\kappa$. Note that for $P,Q\in C(\kbar)$ we only have $\gamma(P)=\gamma(Q)$ if $[2P]=[2Q]$, which implies that $P,Q$ are Weierstrass points on $C$. Hence, the image of $\gamma$ is birational to $C$.
Moreover, since for a Weierstrass point $P$ we do have $[2P]=\kappa$, we see that the identity $0_J\in J$ lies in the image of $\gamma$.

If $J$ is decomposable then there is an isogeny $\Phi\colon J\to E_1\times E_2$ over $k$, where $E_1,E_2$ are elliptic curves over $k$. Let $\pi_1\colon E_1\times E_2 \to E_1$ be the projection on the first factor and write $\Phi_1=\Phi\circ \pi_1$. We claim that $j\circ \Phi_1$ is not constant. If it were, then $\gamma(C)$ would have to lie in the connected component of $\ker(\Phi_1)$ that contains $0_J$. But that is a $1$-dimensional subgroup scheme of $J$, so cannot contain a singular model of a curve of genus $2$. It follows that $\phi=\gamma\circ \Phi_1\colon C \to E_1$ is a non-constant morphism between (complete, non-singular) irreducible curves and hence a finite cover.
\end{proof}

The cover in Lemma~\ref{lem:fincov} is far from unique and the one that the proof constructs is unlikely to be of minimal degree. 
This leads us to consider \emph{optimal} covers, also referred to as \emph{maximal} \cite{MR1085258} and \emph{minimal} \cite{MR1285957} covers.

\begin{defn} We call a finite cover $\phi_1\colon C\to E_1$ \emph{optimal} if for any factorisation
$$C\stackrel{\phi_1}{\to} E_1= C\stackrel{\phi_1'}{\to}D \stackrel{\psi}{\to} E_1$$
we must have $\deg(\phi_1')=\deg(\phi_1)$ or $\deg(\phi_1')=1$.
\end{defn}

It is immediate that any finite cover $\psi:C\to E$, where $C$ is of genus $2$ and $E$ is an elliptic curve, factors through some optimal cover $\phi_1:C\to E_1$.

We follow Kuhn~\cite{MR936803} and Frey-Kani~\cite{MR1085258}. We write $n=\deg(\phi_1)$. We will need our maps to be separable, so we assume that $\charact(k)\nmid n$. We have the induced maps
\[
\phi_1^*\colon E_1\to J \text{ and }\phi_{1,*}\colon J \to E_1.
\]
The optimality of $\phi_1$ implies that $\phi_1^*$ is injective and that $E_2^*:=\ker(\phi_{1,*})$ must be connected and hence an elliptic curve. We write $E_1^*=\phi_1^*(E_1)$.

Since $\phi_{1,*}\circ \phi_1^*=n\cdot\id_{E_1}$ we see that $E_1^*\cap E_2^*=E_1^*[n]$.

We write $\phi_{2,*}\colon J\to J/E_1^* =: E_2$ for the projection. We follow Kuhn's argument \cite{MR936803}. He assumes $k$ is a number field, but his method generalizes. Kuhn proves that if $C$ has a degree $1$ divisor class over $k$ that is invariant under the hyperelliptic involution then there is a cover $\phi_2\colon C\to E_2$ of degree $n$ for which $\phi_{2,*}$ is the corresponding push forward. Furthermore, he shows that if $n$ is odd then such a class exists. For even $n$ he argues that the map, initially defined over an extension of $k$ where $C$ has a Weierstrass point, actually descends to $k$. Note that the kernel of $\phi_{2,*}$ is connected and hence $\phi_2$ is optimal. We call $\phi_2\colon C\to E_2$ a \emph{complimentary cover} to the optimal cover $\phi_1\colon  C\to E_1$.

The maps $\phi_1,\phi_2$ give rise to an isogeny
\[\phi_1^*+\phi_2^*\colon E_1\times E_2\to J,\]
where $\Delta=\ker(\phi_1^*+\phi_2^*)$ is the graph of an isomorphism
\[\alpha\colon E_1[n]\to E_2[n].\]

To characterize the nature of this isogeny, we recall some standard terminology on principally polarized abelian varieties. We follow \cite{MR861974}. We recall that a polarized abelian variety is an abelian variety $A$ equipped with an isogeny $\lambda\colon A\to A^\vee$, where $A^\vee$ is the dual abelian variety of $A$, such that $\lambda$ comes from an ample invertible sheaf on $A_{\kbar}$.
If $\lambda$ is an isomorphism, we say that $(A,\lambda)$ is a \emph{principally polarized abelian variety}. A principal polarization induces, for each $n$ prime to the characteristic, an alternating non-degenerate, bilinear pairing $e_{A[n]}\colon A[n]\times A[n]\to\mu_n$, called a \emph{Weil pairing}.

The main result we need is \cite[Proposition 16.8]{MR861974}, which describes isogenies that respect polarizations. Paraphrased, it yields in our particular situation

\begin{lem} \label{L:nnpoliso}
Let $(A,\lambda_A)$ be a principally polarized abelian variety and let $\Phi\colon A\to B$ be an isogeny with $\ker(\Phi)\subset A[n]$. A necessary and sufficient condition for the existence of a polarization $\lambda_B\colon B\to B^\vee$ such that the diagram
$$\xymatrix{
A\ar[r]^{n\lambda_A}\ar[d]_\Phi&A^\vee\\
B\ar[r]^{\lambda_B}&B^\vee\ar[u]_-{\Phi^\vee}\\
}$$
commutes, is that $\ker(\Phi)$ is \emph{isotropic} with respect to $e_{A[n]}$, which means that $e_{A[n]}$ restricted to $\ker(\Phi)\times\ker(\Phi)$ is trivial.
\end{lem}

If $A$ is $g$-dimensional, then $\deg(n\lambda_A)=n^{2g}$. Since $\deg(\Phi)=\deg(\Phi^\vee)$, a simple degree calculation shows that $\lambda_B$ is principal if and only if $\deg(\Phi)=n^g$. The nondegeneracy of $e_{A[n]}$ implies that in that case $\ker(\Phi)$ is a \emph{maximal} isotropic subgroup. 
\begin{defn}
Let $(A,\lambda_A)$ and $(B,\lambda_B)$ be principally polarized abelian varieties of dimension $g$. We say that an isogeny $\Phi\colon A\to B$ is a \emph{polarized $(n_1,\ldots,n_r)$-isogeny} if $\ker(\Phi)(\kbar)\simeq\ZZ/n_1\ZZ\times\cdots\times\ZZ/n_r\ZZ$  and $\Phi^\vee\circ \lambda_B\circ \Phi=n\lambda_A$, where $n^g=\prod_{i=1}^r n_i$.
\end{defn}

Using \cite[Lemma~16.2]{MR861974} it is straightforward to check that if $\Phi\colon A\to B$ is a polarized $(n,n)$ isogeny between principally polarized abelian surfaces $(A,\lambda_A)$ and $(B,\lambda_B)$, then so is
$\Phi^\vee\colon B^\vee\to A^\vee$ between $(B^\vee,\lambda_B^{-1})$ and $(A^\vee,\lambda_A^{-1})$. Furthermore if $\lambda'$ is another polarization on $B$ such that
$\Phi^\vee\circ \lambda'\circ \Phi=n\lambda_A$ then $\lambda'=\lambda_B$. This can be seen by observing that the N\'eron-Severi group of an abelian variety is torsion-free or, more directly, if $n\lambda'=n\lambda_B$, then $\lambda'-\lambda_B$ maps the connected variety $A$ into a finite variety, so it must be constant $0$. 

\begin{lem}\label{L:optcov-nniso}
Let $C$ be a genus $2$ curve, let $\phi_1\colon C\to E_1$ be an optimal cover of degree $n$ and let $\phi_2\colon C\to E_2$ be a complimentary cover. Then
\[\phi_1^*+\phi_2^*\colon E_1\times E_2\to J\]
is a polarized $(n,n)$-isogeny, with dual isogeny
\[\phi_{1,*}\times\phi_{2,*}\colon J \to E_1\times E_2\]
\end{lem}

\begin{proof}
The duality statement is immediate. To prove that the isogeny is polarized, we just have to verify that
\[(\phi_{1,*}\times\phi_{2,*})\circ(\phi_1^*+\phi_2^*)=(n\,\id_{E_1}\times n\,\id_{E_2})\]
which follows because $\phi_{i,*}\circ\phi_j^*=0$ and $\phi_{i,*}\circ\phi_i^*=n\,\id_{E_i}$ for $(i,j)=(1,2),(2,1)$. Finally, it is an $(n,n)$-isogeny because the kernel, being the graph of an isomorphism $E_1[n]\to E_2[n]$, indeed has the structure $E_1[n](\kbar)\simeq \ZZ/n\ZZ\times\ZZ/n\ZZ$.
\end{proof}

We have
\[
\Delta=\ker(\phi_1^*+\phi_2^*)=\{ (P,\alpha(P)): P \in E_1[n]\}.
\]
For $\Delta$ to be maximally isotropic we need for all $P,Q\in E_1[n]$ that
\[1=e_{(E_1\times E_2)[n]}((P,\alpha(P),(Q,\alpha(Q))=
 e_{E_1[n]}(P,Q) e_{E_2[n]}(\alpha(P),\alpha(Q)),\]
which is precisely the case if $\alpha$ is an anti-isometry.

\begin{defn}\label{D:nnsplitting}
Let $E_1, E_2$ be elliptic curves and let $A$ be a principally polarized abelian surface. Suppose that $\Phi\colon E_1\times E_2\to A$ is a polarized isogeny. We say that $\Phi$ is an \emph{optimal polarized $(n,n)$-splitting}
if $\Delta=\ker(\Phi)$ is the graph of an anti-isometry $\alpha\colon E_1[n]\to E_2[n]$.

A principally polarized abelian surface $A$ equipped with an optimal polarized $(n,n)$-splitting is an \emph{optimally $(n,n)$-split} principally polarized abelian surface.
\end{defn}

\begin{prop}
Let $C$ be a genus $2$ curve over a field $k$ of characteristic $0$. If $\Jac(C)$ is decomposable then for some $n$ it admits an optimal $(n,n)$-splitting.
\end{prop}

\begin{proof}
Lemma~\ref{lem:fincov} guarantees that there is a finite cover $C\to E_1'$ so there is an optimal cover $\phi_1\colon C\to E_1$ as well. Let $n=\deg(\phi_1)$.  Lemma~\ref{L:optcov-nniso} shows that this gives rise to a polarized $(n,n)$-isogeny $E_1\times E_2\to\Jac(C)$ and we have established that its kernel is the graph of an anti-isometry.
\end{proof}

An $(n,n)$-splitting does not have to map to a Jacobian:

\begin{prop}\label{P:n-1isog}
An $(n-1)$-isogeny $\phi\colon E_1\to E_2$ gives rise to an optimal polarized $(n,n)$-splitting
\[
\begin{array}{cccccccc}
 \Phi\colon&E_1&\times&E_2&\to& E_1&\times& E_2\\
&(\;P\;&,&\;Q\;)&\mapsto&(\;\phi^*(Q)+P&,&\phi(P)-Q\;)
\end{array}
\]
where $\phi^*\colon E_2\to E_1$ is the isogeny such that $\phi^*\circ\phi$ is multiplication-by-$(n-1)$.
\end{prop}

\begin{proof}
Note that the restriction $\phi|_{E_1[n]}\colon E_1[n]\to E_2[n]$ yields an anti-isometry.
It is straightforward to check that $\Phi\circ\Phi$ is multiplication-by-$n$ and that $\ker(\Phi)$ consists of points $(P,\phi(P))$, with $P\in E[n]$, so the kernel of $\Phi$ is indeed that graph of an anti-isometry.
\end{proof}


\section{\texorpdfstring{$(2,2)$}{(2,2)}-Split Jacobians}\label{S:22split}

This is a brief description of $(2,2)$-splittings. We believe the results presented here are well-known, but since the construction is central to the rest of the paper and the proofs are simple, we have included them for the convenience of the reader.
See also \cite{MR1913484} and \cite[Ch.~14]{MR1406090}.

\begin{lem}\label{L:2isom-aval}
Let $k$ be a field with $\charact(k)\neq 2$ and let
\[ E_1\colon V^2=f(U)\]
be an elliptic curve over $k$, where $f(U)\in k[U]$ is a monic square-free cubic. Specifying $(E_2,\alpha)$, where $E_2$ is an elliptic curve over $k$ and $\alpha: E_1[2]\to E_2[2]$ is an anti-isometry is equivalent to specifying $a\in k\cup\{\infty\}$ with $f(a)\neq 0$ and $d\in k^\times$ representing an element in $ k^\times/k^{\times2}$ such that
\[
E_2:
\begin{cases}
W^2=-df(U)&\text{ if }a=\infty\\
W^2=d(U-a)f(U)&\text{ otherwise}
\end{cases}
\]
where $0_{E_2}\in E_2(k)$ is the unique point with $U(0_{E_2})=a$ and the anti-isometry is given by
$\alpha(0_{E_1})=0_{E_2}$ and $\alpha( (u,0) )=(u,0)$ for any $(u,0)\in E_1[2](\kbar)\setminus\{0_{E_1}\}$.
\end{lem}

\begin{proof}
First note that any group scheme isomorphism $\alpha\colon E_1[2] \to E_2[2]$ is automatically both an isometry and an anti-isometry and that any scheme isomorphism $\alpha'\colon E_1[2]\setminus\{0_{E_1}\}\to E_1[2]\setminus\{0_{E_1}\}$ can be extended uniquely to an isometry.

We first prove that if $\alpha\colon E_1[2]\to E_2[2]$ is a group scheme homomorphism, then $E_2$ and $\alpha$ can be represented as stated. Note that $U\colon E_1\to\PP^1$ represents the quotient $E_1\to E_1/\langle -1\rangle$ and that it is ramified over exactly
$U(E_1[2])=\{f(U)=0\}\cup\{\infty\}$.
Similarly, we have $U'\colon E_2\to E_2/\langle -1\rangle$ and $\alpha$ induces a scheme isomorphism
$\gamma\colon \{f(U)=0\}\to U'(E_2[2]\setminus\{0_{E_2}\}$. Since this is an isomorphism of \'etale degree $3$ subschemes of $\PP^1$, it extends uniquely to an isomorphism $\PP^1\to \PP^1$. Hence $\gamma^{-1}\circ U'\colon  E_2\to \PP^1$ is a degree $2$ cover ramified over $\{f(U)=0\}$ and some fourth point $\gamma^{-1}(U'(0_{E_2}))=a$ (hence $f(a)\neq 0$). It follows that $E_2$ admits a model as stated and that $\alpha$ is a map as advertised.

Conversely, it is clear that as long as $f(a)\neq 0$, the model for $E_2$ describes an elliptic curve and $\alpha$ describes a scheme isomorphism $E_1[2]\to E_2[2]$ sending $0_{E_1}$ to $0_{E_2}$, so it does define an anti-isometry.
\end{proof}

\begin{thm}\label{T:2x2model} Let $k$ be a field with $\charact(k)\neq 2$. Let $E_1, E_2$ be elliptic curves
given by models
\[\begin{aligned}
 E_1\colon& V^2=f(U)\\
 E_2\colon& W^2=d(U-a)f(U)
  \end{aligned}
\]
and let $\alpha\colon E_1[2]\to E_2[2]$ be the isometry induced by the identification $U(E_1[2]\setminus\{0_{E_1}\})=U(E_2[2]\setminus\{0_{E_2}\})$.
If $a\neq\infty$ then the fiber product $C_2= E_1\times_{\PP^1_U} E_2$ is a curve of genus $2$ admitting a model
\[C_2\colon Y^2=f(\frac{1}{d}X^2+a),\]
where the double covers $\phi_1\colon C_2\to E_1$ and $\phi_2\colon C_2\to E_2$ are induced by the relations
\[\begin{aligned}
U=\frac{1}{d}X^2+a,\quad V=Y,\quad W=XY.
  \end{aligned}\]
Furthermore, the isogeny
\[\phi_1^*+\phi_2^*\colon E_1\times E_2 \to \Jac(C_2)\]
is the $(2,2)$-splitting corresponding to $\alpha$.
\end{thm}
\begin{proof}
That $C_2$ is a model of the fiber product of $E_1$ and $E_2$ over the $U$-line can be verified immediately. If we establish that $\phi_1$ is an optimal cover and that $\phi_2$ is a complimentary cover then Lemma~\ref{L:optcov-nniso} establishes that $\phi_1^*+\phi_2^*$ is a $(2,2)$-splitting.  Optimality follows because $\phi_1$ and $\phi_2$ are of prime degree. It follows that $\phi_1^*: E_1\to \Jac(C_2)$ is injective.

To show that $\phi_2$ is complimentary we need that $\phi_{2,*}\circ\phi_1^*=0$. But these are maps that come from a fiber product, so we can compute the composition by taking a divisor on $E_1$, push it down to $\PP^1_U$ and pull it back to $E_2$. Since we map through a $\PP^1$, any degree $0$ divisor must map into the principal class on $E_2$, which establishes that $\phi_{2,*}\circ\phi_1^*=0$.

It is straightforward to check that $\phi_1^*+(\phi_2^*\circ\alpha):E_1[2]\to\Jac(C_2)$ is zero and hence that the kernel of $\phi_1^*+\phi_2^*$ is indeed the graph of $\alpha$.
\end{proof}

\begin{defn}
Let $E$ be an elliptic curve over a separable quadratic extension $L/k$. We write $\Re_{L/k}(E)$ for the \emph{Weil restriction of scalars} of $E$ with respect to $L/k$, in the sense of \cite[\S7.6]{MR1045822}.
\end{defn}

For the purposes here, it is sufficient to know that $A=\Re_{L/k}(E)$ is an abelian surface over $k$ that over $L$ is isomorphic to $E\times {E}^\sigma$, where $\sigma$ is a non-trivial automorphism of $L$ over $k$. The product polarization on the latter descends to a $k$-rational principal polarization on $A$.

\begin{prop}
 \label{T:2x2modeldegen} Let $k$ be a field with $\charact(k)\neq 2$. Let $E$ be an elliptic curve over $k$, let $d\in k^\times$ represent a class in $k^\times/k^{\times 2}$ and let $\alpha: E[2]\to E^{(d)}[2]$ be the obvious isometry. Let $\Delta\subset E[2]\times E^{(d)}[2]$ be the graph of $\alpha$. Then
\[
(E\times E^{(d)})/\Delta =
\begin{cases}
E\times E&\text{ if $d$ is a square}\\
\Re_{k(\sqrt{d})/k)}(E)&\text{ otherwise}
\end{cases}
\]
\end{prop}
\begin{proof}
If $d$ is square, Proposition~\ref{P:n-1isog} applies with $n=2$ and we find the $(2,2)$-isogeny given by $\Phi\colon(P,Q)\mapsto(P+Q,P-Q)$.

If $d$ is not a square, the first case at least gives us a description of $\Phi$ over $k(\sqrt{d})$. We just have to check that $\Phi$ descends to a morphism over $k$ with the twisted Galois actions on domain and codomain. Both $(E\times E^{(d)})(\kbar)$ and $\Re_{k(\sqrt{d})/k)}(E)(\kbar)$ are isomorphic to $E(\kbar)\times E(\kbar)$ as groups, but have twisted Galois actions. Let $\chi_d\colon \Gal(\kbar/k)\to \{\pm 1\}$ be the quadratic character belonging to $k(\sqrt{d})/k$. The Galois action on $E(\kbar)\times E(\kbar)$ corresponding to $E\times E^{(d)}$ is
\[
 (P,Q)^\sigma=(P^\sigma,\chi_d(\sigma)Q^\sigma)
\]
and the action corresponding to $\Re_{k(\sqrt{d})/k)}(E)$ is
\[
 (P,Q)^\sigma=\begin{cases}
               (P^\sigma,Q^\sigma)&\text{ if } \chi_d(\sigma)=1\\
               (Q^\sigma,P^\sigma)&\text{ if } \chi_d(\sigma)=-1.
              \end{cases}
\]
We want to test that the isogeny $\Phi\colon E(\kbar)\times E(\kbar)\to E(\kbar)\times E(\kbar)$ defined by $(P,Q)\mapsto (P+Q,P-Q)$ descends to $k$ when we twist domain and codomain to $E\times E^{(d)}$ and $\Re_{k(\sqrt{d})/k)}(E)$ respectively. So we must establish that
$(\Phi(P,Q))^\sigma=\Phi((P,Q)^\sigma)$ for all $\sigma\in\Gal(\kbar/k)$, with the appropriately interpreted twisted action. It is immediate that this is the case if $(\sqrt{d})^\sigma=\sqrt{d}$. In the other case we verify that
\[
\begin{split}
(\Phi(P,Q))^\sigma=(P+Q,P-Q)^\sigma=(P^\sigma-Q^\sigma,P^\sigma+Q^\sigma)=&\\
 (P^\sigma+\chi_d(\sigma)Q^\sigma,P^\sigma-\chi_d(\sigma)Q^\sigma)=
\Phi(P^\sigma,\chi_d(\sigma)Q^\sigma)=&\,\Phi((P,Q)^\sigma).
\end{split}
\]
This confirms that the isogeny is indeed defined over $k$. It also shows that the product polarization on $E\times E$ over $k(\sqrt{d})$ descends to a principal polarization on $\Re_{k(\sqrt{d})/k)}(E)$ over $k$ such that $\Phi$ is a polarized $(2,2)$-isogeny.
\end{proof}


\section{Polarized \texorpdfstring{$(2,2)$}{(2,2)}-Isogenies on Jacobians of genus 2 curves}\label{S:22maps}

The purpose of this section is to describe polarized $(2,2)$-isogenies between Jacobians of genus $2$ curves.
Such isogenies are called \emph{Richelot isogenies}. After giving an explicit description of the $2$-torsion, we review the classical description of Richelot isogenies over algebraically closed base fields, or more generally, base fields with sufficient roots.
Most of the material presented here is already known, see \cite{MR970659}, \cite[Chapter 8]{SmithThesis},  \cite[Chapter 9]{MR1406090}, or \cite[Section~4]{MR1736231}.
The new contribution is Proposition~\ref{T:twist}, where we determine the appropriate twist of the codomain for non-algebraically closed base fields.

Let $k$ be a field of odd characteristic, let $\kbar$ be an algebraic closure of $k$ and let $C$ be a curve of genus $2$ over $k$.
Then $C$ admits a model of the form
\begin{equation}\label{E:2x2model}
C\colon Y^2=f(X)=f_6 X^6+ f_5 X^5+ \cdots +f_1X+f_0,
\end{equation}
where $f(X)\in k[X]$ is a square-free polynomial of degree $5$ or $6$.  If $k$ has at least $6$ elements, then we can assume that $f_6\neq 0$. This excludes some curves over $k=\FF_3,\FF_5$ from our considerations. In fact, such curves have at least $4$ rational Weierstrass points, which forces the Galois structure of the kernel of Richelot isogenies defined over $k$  to be of the type that is already covered by the existing literature. Note that $(f_6Y)^2=f_6^2f(X)$ is also a model of $C$ over $k$, so it is not a restriction to insist that the leading coefficient is a cube. We assume that $f_6 = q_2^3$ for some $q_2\in k$.

First we describe $\Jac(C)[2]$ and its maximal isotropic subgroups.
Let $w_1,\ldots,w_6$ be the roots of $f(X)$ in $\kbar$. The Weierstrass points of $C$ are exactly $T_i=(w_i,0)$. The non-zero two-torsion points in $\Pic^0(C/\kbar)$ are exactly the divisor classes $T_{\{i,j\}}=[T_i-T_j]=[T_j-T_i]$, and the Weil-pairing is given by
$$(T_{\{i,j\}},T_{\{k,l\}})_2=(-1)^{\#\{i,j,k,l\}}.$$
Let $J=\Jac(C)$. The maximal isotropic subgroups of $J[2]$ are exactly of the form
$$\{0,T_{\{i_1,i_2\}},T_{\{i_3,i_4\}},T_{\{i_5,i_6\}}\},$$
where the indices are given by a partition $\{\{i_1,i_2\},\{i_3,i_4\},\{i_5,i_6\}\}$ of $\{1,\ldots,6\}$ into three disjoint pairs. For ease of notation, we assume that $(i_1,\ldots,i_6)=(1,\ldots,6)$.  This data corresponds to specifying a factorization
\begin{equation}\label{E:quadfactor}
F_j(X)=q_2 X^2+ q_{1,j} X+ q_{0,j}=q_2 (X-w_{2j-1})(X-w_{2j})
\end{equation}
such that
$$f(X)=F_1(X)F_2(X)F_3(X).$$
We say that $\{F_1(X),F_2(X),F_3(X)\}\subset\kbar[X]$ is a \emph{quadratic splitting} of $f$.
We say that $\{F_1(X),F_2(X),F_3(X)\}$ is a quadratic splitting \emph{over $k$} if it is stable under $\Gal(\kbar/k)$. The $F_i(X)$ do not have to be individually defined over $k$.

\begin{lem}
Let $C$ be a curve of genus $2$ over a field $k$ of odd characteristic with $\#k>5$. Suppose $\Delta\subset\Jac(C)[2]$ is a maximal isotropic subgroup scheme over $k$. Let $L$ be the coordinate ring of $\Delta\setminus\{0\}$. Then there is a quadratic polynomial $Q(X)\in L[X]$ such that $C$ admits a model of the form
\begin{equation}\label{E:normform}
C: Y^2=f(X)=\Norm_{L[X]/k[X]}(Q(X)).
\end{equation}
Conversely, for any cubic \'etale algebra $L/k$, any such representation gives rise to a maximal isotropic subgroup scheme $\Delta\subset\Jac(C)[2]$ with $\Delta\setminus\{0\}=\mathrm{Spec}(L)$.
\end{lem}

\begin{proof}
We choose a model of the form \eqref{E:2x2model} with $f_6=q_2^3$. We label the roots $w_1,\ldots,w_6$ of $f(X)$ in $\kbar$ such that
\[\Delta(\kbar)=\{0,T_{\{1,2\}},T_{\{3,4\}},T_{\{5,6\}}\}\]
Let $F_j(X)$ be defined as \eqref{E:quadfactor}. The group $\Gal(\kbar/k)$ acts by permutation on $\{T_{\{1,2\}},T_{\{3,4\}},T_{\{5,6\}}\}$ and the identification $F_j(X)\mapsto T_{\{2j-1,2j\}}$ is Galois-covariant, so
$\{F_1(X),F_2(X),F_3(X)\}$ is a quadratic splitting of $f(X)$ over $k$.
It follows that there is a polynomial $Q(X)\in L(X)$ that maps to each of the $F_j$ under the three $k$-algebra homomorphisms $L\to\kbar$. This yields that $C$ is indeed of the form \eqref{E:normform}.

For the converse, note that the three images $F_j(X)$ of $Q(X)$ under the three maps $L[X]\to \kbar[X]$ give rise to a quadratic splitting $\{F_1(X),F_2(X),F_3(X)\}$ of $f(X)$ over $k$ and hence to a maximal isotropic subscheme $\Delta\subset \Jac(C)[2]$ over $k$ with $\Delta\setminus\{0\}=\mathrm{Spec}(L)$.
\end{proof}

Next we describe the codomain of a Richelot-isogeny.
Suppose that $\Delta\subset\Jac(C)[2]$ is a maximal isotropic subgroup scheme over $k$ and let $\{F_1(X),F_2(X),F_3(X)\}$ be the corresponding quadratic splitting.
We will describe the principally polarized abelian surface $B=\Jac(C)/\Delta$ when it is a Jacobian itself.
We define the \emph{determinant} of the quadratic splitting to be
\begin{equation}\label{E:delta}
\delta=\det
\begin{pmatrix}
 q_{0,1}& q_{1,1} & q_{2}\\
 q_{0,2}& q_{1,2} & q_{2}\\
 q_{0,3}& q_{1,3} & q_{2}\\
\end{pmatrix}
\end{equation}
(see \cite[~page 117]{SmithThesis} or \cite[~page 89]{MR1406090}). 
If $\delta = 0$ then we say the quadratic splitting $\{F_1(X), F_2(X), F_3(X)\}$ is \emph{singular} In this case $B$ is a product of elliptic curves over $\kbar$.  Otherwise, $B$ is the Jacobian of a genus 2 curve over $\kbar$ and we say $\{F_1(X), F_2(X), F_3(X)\}$ is \emph{nonsingular}.

For a non-singular quadratic splitting, the following classical construction gives a curve $\tilde{C}_1$ such that $B=\Jac(\tilde{C}_1)$ over $\kbar$.
Suppose $\{F_1(X), F_2(X), F_3(X)\}$ is nonsingular.  Then for $(i,j,k)=(1,2,3),(2,3,1),(3,1,2)$ we define
$$G_i(X)=\delta^{-1}\det
\begin{pmatrix}
  \frac{d}{dX}F_j(X)& \frac{d}{dX}F_k(X)\\
  F_j(X)            & F_k(X)
\end{pmatrix}
$$
It is straightforward to check that $\{G_1(X),G_2(X),G_3(X)\}\subset\kbar[X]$ is again stable under $\Gal(\kbar/k)$.
For $d\in k^*$, we consider the curve
\begin{equation}\label{eq:Cdtilde}
\tilde{C}_d\colon d\tilde{Y}^2 = g(\tilde{X}) = G_1(\tilde{X})G_2(\tilde{X})G_3(\tilde{X}).
\end{equation}

\begin{lem} If $\delta\neq 0$ then
the polynomial $g$ is squarefree of degree 5 or 6.
\end{lem}

\begin{proof}
This follows by direct computation; see \cite[Page 122]{SmithThesis}.  
\end{proof}

We are now ready to review the Richelot isogeny.
From \cite[~Theorem 8.4.11]{SmithThesis} or \cite[~Section 3.1]{MR970659}
we know that over $\kbar$ we have $B=\Jac(\tilde{C}_1)$ and that the isogeny is described by a \emph{Richelot correspondence} defined by a curve $\Gamma_d\subset C\times \tilde{C_d}$ over $\kbar$ given by
$$\Gamma_d\colon\left\{\begin{array}{rcl}
F_1(X)G_1(\tilde{X})+F_2(X)G_2(\tilde{X})&=&0\\
F_1(X)G_1(\tilde{X})(X-\tilde{X})&=& \sqrt{d}\,\tilde{Y}Y\\
F_2(X)G_2(\tilde{X})(X-\tilde{X})&=&-\sqrt{d}\,\tilde{Y}Y\\
\end{array}\right.$$
The curve $\Gamma_d$ covers both $C$ and $\tilde{C}_d$. The Richelot isogeny can be computed by taking divisor classes on ${C}$, pulling back to $\Gamma_d$ and then pushing down to $\tilde{C}_d$.

There are two cases where it is easy to see for which twist $d$ we have $\Jac(\tilde{C}_d)=B$.

First, if $F_1,F_2,F_3\in k[X]$ and $d=1$, then $\Gamma_d$ is defined over $k$ and hence  $B=\Jac(\tilde{C}_1)$ over $k$.

Second, if $F_1$ and $F_2$ are quadratic conjugate, say over an extension $k(\sqrt{d})$, then $F_3$ is necessarily defined over $k$. Then the set of defining equations for $\Gamma_d$ is $\Gal(\kbar/k)$-stable, and hence $\Gamma_d$ is defined over $k$. Since over $\kbar$, the curves $\tilde{C}_d$ and $\Gamma_d$ are isomorphic to $\tilde{C}_1$ and $\Gamma_1$, it follows from the above discussion that $\Gamma_d$ describes a correspondence giving rise to an isogeny $\Jac(C)\to\Jac(\tilde{C}_d)$ of the desired type.
Note that $d=\disc(L)$.

\begin{prop}\label{T:twist}
Let $C$ be a genus $2$ curve as in \eqref{E:normform}.
Let $\Delta\subset\Jac(C)[2]$ be the maximal isotropic subgroup scheme over $k$
with $\delta\neq 0$ and $\Delta\setminus\{0\}=\mathrm{Spec}(L)$. Let $d=\disc(L)$. Then $\Jac(C)/\Delta=\Jac(\tilde{C}_d)$.
\end{prop}
\begin{proof}
The cases where $\Gal(\kbar/k)$ acts non-transitively on $\Delta(\kbar)\setminus\{0\}$ have been dealt with above. For the general case we consider a generic model. We will prove it there and all special cases follow by specialization.

We consider the field $K=k(h_0,h_1,h_2,q_{i,j})$ with $i,j\in\{0,1,2\}$ and let $L=K[T]/(T^3+h_2T^2+h_1T+h_0$ and let $Q(X)\in L[X]$ be defined by
$$Q=\sum_{i,j=0}^2 q_{i,j} T^j\,X^i.$$
We now consider the curve $C\colon Y^2=f(X)=\Norm_{L[X]/k[X]}(Q(X))$ over $K$. 

We have that $L/K$ is a cubic extension with Galois closure $L(\sqrt{d})$ over $K$. Using the discussion above, we know that $B=\Jac(\tilde{C}_d)$ over $L$.
However, we know that $B$ itself is defined over $K$ as a principally polarized variety, so $B$ must be some twist of $\Jac(\tilde{C}_d)$ that trivializes over the cubic extension $L$. However we have $\Aut_{\overline{K}}(\Jac(\tilde{C}_d))=\Aut_{\overline{K}}(\tilde{C}_d)=\{\pm 1\}$, so both only have quadratic twists. It follows the twist must be trivial.

Specialization now yields that for any curve $C$ of the stated form, a polarization preserving isomorphism $\Jac(C)/\Delta\simeq \Jac(\tilde{C}_d)$ over $k$ exists.
\end{proof}


\section{\texorpdfstring{$(4,4)$}{(4,4)}-split principally polarized abelian surfaces}\label{S:4x4to2x2}

Let $J$ be a principally polarized abelian surface with an optimal $(4,4)$-splitting $\Phi_4\colon E_1\times E_2\to J$ such that the kernel $\Delta_4\subset E_1[4]\times E_2[4]$ is the graph of an anti-isometry $\alpha_4\colon E_1[4]\to E_2[4]$.
Since $E_i[2]\subset E_i[4]$, we also have $\alpha_2=\alpha_4|_{E_1[2]}\colon E_1[2]\to E_2[2]$.
The subgroup $\Delta_2=\Delta_4\cap (E_1\times E_2)[2]$ is the graph of $\alpha_2$, so we see that $\Phi_4$ factors through an optimal $(2,2)$-splitting $E_1\times E_2\to A=(E_1\times E_2)/\Delta_2$.
We use the principal polarizations to identify $E_1\times E_2, A, J$ with their duals. We obtain
the diagram
\begin{equation}\label{E:44split}
\xymatrix@C4em{
   E_1\times E_2 \ar[r]^-{2\lambda_{E_1\times E_2}} \ar[d]^{\Phi_2} \ar@/_3ex/[dd]_{\Phi_4}&
   (E_1\times E_2)^\vee \ar[r]^{2} &
   (E_1\times E_2)^\vee
\\
   A\ar[r]^{\lambda_A} \ar[d]^{\Psi}&
   A^\vee\ar@{-->}[r]^{2} \ar[u]^-{\Phi_2^\vee}&
   A^\vee\ar[u]^{\Phi_2^\vee}
\\
   J\ar[rr]^{\lambda_{J}} &&
   J^\vee \ar[u]^{\Psi^\vee}\ar@/_3ex/[uu]_{\Phi_4^\vee}
}\end{equation}
in which we want to establish that also with the addition of the dashed arrow, the diagram is commutative. To lighten our notation, we avoid explicitly referring to the polarizations as much as possible. To this end we introduce the shorthand notation
\[\begin{split}
\Psi^*=&\lambda_A^{-1}\circ \Psi^\vee\circ\lambda_{J}, \\
\Phi_2^*=&\lambda_{E_1\times E_2}^{-1}\circ \Phi_2^\vee\circ\lambda_A,\\
\Phi_4^*=&\lambda_{E_1\times E_2}^{-1}\circ \Phi_4^\vee\circ\lambda_{J}.\\
  \end{split}
\]

\begin{lem}\label{T:kernels}
The isogeny $\Psi\colon A\to J$ is a polarized $(2,2)$-isogeny. Furthermore, $\ker(\Psi)\cap\ker(\Phi_2^*)=\{0\}$.
\end{lem}

\begin{proof}
It follows from \cite[Lemma~16.2c]{MR861974} that for $p,q\in (E_1\times E_2)[4]$, we have that $e_{(E_1\times E_2)[4]}(p,q)=e_{A[2]}(\Phi_2(p),\Phi_2(q))$. Hence we see that $\ker(\Psi)=\Phi_2(\Delta_4)\subset A[2]$ is maximal isotropic, so by Lemma~\ref{L:nnpoliso} there is a principal polarization $\lambda'\colon J\to J^\vee$ such that $2\lambda_A=\Psi^\vee\circ \lambda'\circ\Psi$. 
It follows that
\[\Phi_2^\vee\circ\Psi^\vee\circ \lambda'\circ\Psi\circ\Psi_2=4\lambda_{E_1\times E_2}=\Psi_4^\vee\circ\lambda_{J}\circ\Psi_4,\]
so the image of $(\lambda'-\lambda_{J})\circ\Psi_4$ is contained in $\ker(\Psi_4^\vee)$, which is finite. On the other hand, $\Psi_4$ is surjective and $J$ is connected, so $\lambda'-\lambda_{J}$ is constant and hence $\lambda'=\lambda_{J}$. This establishes that $\Psi$ is indeed a polarized $(2,2)$-isogeny.

In order to see that $\ker(\Psi)\cap\ker(\Phi_2^*)=\{0\}$, note that $\Phi_2$ is injective on $E_1[2]\times\{0\}$ and maps it onto $\ker(\Phi_2^*)$, because $\Phi_2^*\circ\Phi_2=2$.
Since $\Psi\circ\Phi_2$ is injective on $E_1[4]\times\{0\}$, it follows that $\Psi$ is also injective on $\Phi_2(E_1[2]\times\{0\})=\ker(\Phi_2^*)$. This shows that $\ker(\Psi)\cap\ker(\Psi_2^*)=\{0\}$.
\end{proof}

In fact, whether $\Psi\circ\Phi$ is a $(4,4)$-isogeny is completely determined by $\ker(\Psi)\cap\ker(\Phi^*)$.

\begin{lem}\label{L:44iso-kernels}
Let $A,B,J$ be polarized abelian surfaces and suppose that $\Phi^*\colon A\to B$ and $\Psi\colon A\to J$ are polarized $(2,2)$-isogenies. Then $\Psi\circ\Phi\colon B\to J$ is a polarized $(4,4)$-isogeny if and only if $\ker(\Psi)\circ\ker(\Phi^*)=\{0\}$. It is a $(4,2,2)$-isogeny if and only if $\ker(\Psi)\cap\ker(\Phi^*)\simeq \ZZ/2\ZZ$ and it is a $(2,2,2,2)$-isogeny if and only if $\ker(\Psi)=\ker(\Phi^*)$.
\end{lem}

\begin{proof}
It is immediate that $\Psi\circ\Phi$ is a polarized isogeny. The nature of the isogeny can be read off from the kernel, so we investigate what these isogenies do on the $4$-torsion. The isogenies we consider fit in the following commutative diagram.
\[\xymatrix{
B[4]\ar[rr]^2\ar[dr]_{\Phi}&&B[4]\ar[rr]^2\ar[dr]_-{\Phi}&&B[4]\\
&A[4]\ar[rr]^2\ar[dr]_{\Psi}\ar[ur]^{\Phi^*}&&A[4]\ar[ur]_{\Phi^*}\\
&&J[4]\ar[ur]_{\Psi^*}
}\]
Each of $B[4](\kbar),A[4](\kbar),J[4](\kbar)$ is isomorphic to $(\ZZ/4\ZZ)^4$ as a $\ZZ$-module. We normalize choice of basis such that the Weil pairing on each is given by
\[e(\underline{v},\underline{w})=\underline{v}^T
\begin{pmatrix}
0&0&0&1\\0&0&1&0\\0&-1&0&0\\-1&0&0&0
\end{pmatrix}\underline{w}\]
The following are matrices that correspond to polarized $(2,2)$-isogenies,
\[
M=\begin{pmatrix}
   2&0&0&0\\0&2&0&0\\0&0&1&0\\0&0&0&1
  \end{pmatrix},\;
M^*=\begin{pmatrix}
   1&0&0&0\\0&1&0&0\\0&0&2&0\\0&0&0&2
  \end{pmatrix},\;
N=\begin{pmatrix}
   2&0&0&0\\0&1&0&0\\0&0&2&0\\0&0&0&1
  \end{pmatrix},\;
N^*=\begin{pmatrix}
   1&0&0&0\\0&2&0&0\\0&0&1&0\\0&0&0&2
  \end{pmatrix},
\]
where $MM^*=NN^*=2\id$. It is straightforward to check that $\ker(NM)\simeq (\ZZ/4\ZZ)\times(\ZZ/2\ZZ)^2$ and that $\ker(MM)=(\ZZ/4\ZZ)^2$. Correspondingly, we find $\ker(N)\cap\ker(M^*)\simeq (\ZZ/2\ZZ)$ and that $\ker(M)\cap\ker(M^*)=0$, so only the $(4,4)$-isogeny gives rise to trivially intersecting kernels.

It remains to check that these isogenies represent all possibilities. To that end, we observe that a $(2,2)$ isogeny is determined up to isomorphism by its kernel, and that there are $15$ maximal isotropic subgroups in $(\ZZ/2\ZZ)^4$.
It is straightforward to check that if we choose two such subgroups $K_1,K_2$ then there is a transformation $T\in\Sp_4(\ZZ/2\ZZ)$ such that $(TK_1,TK_2)$ is one of
\[\big(\ker(M^*)[2],\ker(M^*)[2]\big),\big(\ker(M^*)[2],\ker(M)[2]\big),\big(\ker(M^*)[2],\ker(N)[2]\big).\]
This shows that by choice of basis, one can ensure that the isogenies considered are indeed represented by the matrices given.
\end{proof}

Lemma~\ref{L:2isom-aval} and Theorem~\ref{T:2x2model} imply that if $j(E_1)\neq j(E_2)$ then $A=\Jac(C_2)$ for some genus $2$ curve $C_2$. Similarly, we expect $J$ to be a Jacobian outside some special conditions. Remark~\ref{R:4x4split-3iso} and Proposition~\ref{P:44degen}
describe such special conditions.
In fact, Theorem~\ref{T:44splitclassification} establishes that these describe all cases where $J$ is not a Jacobian.

\begin{rem}\label{R:4x4split-3iso}
A $3$-isogeny $\phi\colon E_1\to E_2$ induces an anti-isometry $\alpha_4\colon E_1[4]\to E_2[4]$.
By Proposition~\ref{P:n-1isog}, we have $J=E_1\times E_2$ in this case.

If $j(E_1)\neq 0$ then $j(E_2)\neq j(E_1)$, so  $A=\Jac(C_2)$.
The 3-isogeny $-\phi$ gives rise to the same $A,J$ but a different $(4,4)$-splitting, so we find that $C_2$ is a genus $2$ curve that is a double cover of $E_1$ and of $E_2$ in $3$ different ways, see also \cite{MR1913484}. If $j(E_1)=j(E_2)=0$ we find that $A$ is not a Jacobian.
\end{rem}

\begin{prop}\label{P:44degen}
Let $E$ be an elliptic curve with  discriminant $D$. Suppose that $E$ has a rational point $T_1\in E[2](k)$ of order two.

If $D$ is a square then $E[2](k)=\{0,T_1,T_2,T_3\}$ and $E$ has three $2$-isogenies $\phi_i\colon E\to E/\langle T_i\rangle$.
The morphism
\begin{equation}\label{E:degen44splitting}
\begin{array}{cccccccc}
 \Phi\colon& E&\times& E &\to& E/\langle T_2\rangle &\times&E/\langle T_3\rangle\\
&(\;P\;&,&\;Q\;)&\mapsto& (\;\phi_2(P+Q)\;&,&\;\phi_3(P-Q)\;)
\end{array}
\end{equation}
is an optimal $(4,4)$-splitting.

If $D$ is not a square then $E[2](k(\sqrt{D}))=\{0,T_1,T_2,T_3\}$ and \eqref{E:degen44splitting} descends to a $(4,4)$-splitting over $k$ denoted by
\[\Phi'\colon E\times E^{(D)} \to \Re_{k(\sqrt{D})/k}(E/\langle T_2\rangle).\]
\end{prop}

\begin{proof}
If $E$ has square discriminant then we know that the extension generated by $E[2](\kbar)$ is either $k$ or a cyclic cubic extension. The assumption that $T_1\in E[2](k)$ implies it is the former.

In this case, it is clear that $\Phi$ is an isogeny of degree $16$, defined over $k$. To check that $\Phi$ is an optimal $(4,4)$-splitting, we determine $\ker(\Phi)(\kbar)$. Suppose that $(P,Q)\in \ker(\Phi(\kbar)$. Then $Q=P$ if $2P=0$ or $2P=T_2$ and $Q=-P$ if $2P=T_3$ or $2P=T_2+T_3$.

We fix generators $E[4](\kbar)=\langle P_2,P_3\rangle$ with $2P_2=T_2$ and $2P_3=T_3$. Then $Q=\alpha(P)$ where $\alpha\colon E[4](\kbar)\to E[4](\kbar)$ is defined by $P_2\mapsto P_2$ and $P_3\mapsto -P_3$. This is indeed an anti-isometry.

If $D$ is a non-square then we can still define $\Phi$ over $k(\sqrt{D})$. The domain and codomain of $\Phi'$ are isomorphic over $k(\sqrt{D})$ to those of $\Phi$. Checking that $\Phi$ descends to $\Phi'$ over $k$ is a straightforward exercise in checking Galois actions.
\end{proof}


\section{2-level structure on curves of genus 2}\label{S:2Torsion}

The main result in this section is the proof of Theorem~\ref{T:galchar}. We also discuss how the result can be interpreted in terms of moduli spaces of genus $2$ curves.

\begin{proof}[Proof of Theorem~\ref{T:galchar}]
Let $C: Y^2=f(X)$ be a curve of genus $2$ over a field $k$ of odd characteristic and let $J=\Jac(C)$.
Recall from Section~\ref{S:22maps} that $J[2](\kbar)$ can be represented by differences of Weierstrass points of $C$. It follows that the action of $\Gal(\kbar/k)$ on $J[2](\kbar)$, which is through $\Sp_4(\ZZ/2\ZZ)$, factors through the action on the $6$ Weierstrass points, which is through $S_6$.
This yields a homomorphism $S_6\to\Sp_4(\ZZ/2\ZZ)$ and it is straightforward to check that it is an isomorphism.

We have also seen that maximal isotropic subgroups of $J[2]$ correspond to quadratic splittings of $f(X)$. It is straightforward to check that $S_6$ acts transitively on the quadratic splittings of $f(X)$. If $J[2]$ has a polarized $(2,2)$-isogeny over $k$, then $f(X)$ must have a Galois-stable quadratic splitting. We have
$$\mathrm{Stab}_{S_6}(\{\{1,2\}, \{3,4\},\{5,6\}\})\simeq (C_2)^3\rtimes S_3$$
Furthermore, the remaining $14$ quadratic splittings have two orbits under $(C_2)^3\rtimes S_3$, one of length $6$ and one of length $8$.
If $J[2]$ is to have two $k$-rational polarized $(2,2)$-isogenies then $\Gal(\kbar/k)$ should act through the stabilizer subgroup of a representative of one of those orbits.
If we pick a stabilizer subgroup of the first orbit, we obtain
\[ C_2\times C_4 =\langle (12),(34)(56),(35),46)\rangle\]
stabilizing $3$ quadratic splittings
\begin{equation}\label{E:C2xV4}
  	\big\{\{1,2\},\{3,4\},\{5,6\}\big\},\,
  	\big\{\{1,2\},\{3,5\},\{4,6\}\big\},\,
  	\big\{\{1,2\},\{3,6\},\{4,5\}\big\}
\end{equation}
and for the second orbit we obtain
\[\tilde{S}_3=\langle (135)(246),(12)(36)(45)\rangle
\]
stabilizing
\begin{equation}\label{E:S3}
	\big\{\{1,2\},\{3,4\},\{5,6\}\big\},\,
  	\big\{\{1,4\},\{2,5\},\{3,6\}\big\},\,
  	\big\{\{1,6\},\{2,3\},\{4,5\}\big\}.
\end{equation}
Combining this with Lemma~\ref{L:44iso-kernels} yields that \eqref{E:C2xV4} corresponds to isogenies that combine to $(4,2,2)$-isogenies and that \eqref{E:S3} corresponds to isogenies that combine to $(4,4)$-isogenies, completing the proof of Theorem~\ref{T:galchar}.
\end{proof}

\begin{lem}\label{T:C4andC2} Let $k$ be a field with $\charact(k)\neq 2$.
Let $\Phi_4\colon E_1\times E_2 \to J$ be an optimal $(4,4)$-splitting over $k$ that factors through the $(2,2)$-splitting $\Phi_2\colon E_1\times E_2\to A$. Suppose that $A=\Jac(C_2)$ where $C_2$ is a curve of genus $2$. Then $C_2$ admits a model of the form
$$C_2\colon Y^2=g(X)=f(X^2)=c_3X^6+c_2X^4+c_1X^2+c_0,$$
such that $g(X)$ and $f(X)$ have the same splitting field and $\Gal(g)$ is isomorphic to $\tilde{S_3}$ as permutation group.
\end{lem}
\begin{proof}
By Theorem~\ref{T:2x2model}, the curve $C_2$ admits a model of the given form,
where $V^2 = f(U)$ is a model of $E_1$. It remains to prove that $g(X)$ and $f(X)$ have the same splitting field.

Let $L$ denote the splitting field of $g$ and let $K$ denote the splitting field of $f$.  Then $K$ is an extension of $k$, and either $L$ is a degree two extension of $K$ or $L = K$.
By Theorem~\ref{T:galchar} we know that $\Gal(L/k) \leq \tilde{S_3}$.

The three kernels of the $(2,2)$-isogenies that are fixed by $\tilde{S_3}$ are given by the partitionings in \eqref{E:S3}.  A simple verification shows that $\tilde{S_3}$ acts faithfully on each of these kernels.  In particular, if $\set{0, T_1, T_2, T_3}$ is the kernel of the polarized $(2,2)$-isogeny $\Jac(C_2) \rightarrow E_1 \times E_2$, then $\tilde{S_3}$ has the canonical $S_3$-action on $\set{T_1, T_2, T_3}$.  Thus, $\tilde{S_3}$ has the usual $S_3$ action on the roots of $f$. It follows $f$ and $g$ have the same splitting field.
\end{proof}

While the proof of and the condition given in Theorem~\ref{T:galchar} are Galois-theoretic, specifying multiple $(2,2)$-isogenies on $\Jac(C)$ amounts to specifying partial level structure, so one expects that the structure of the result is reflected in covers of moduli spaces as well. We will sketch how one can obtain such a formulation.

Let $k$ be a field of characteristic different from $2$. Any curve of genus $2$ can be obtained by specializing $(f_0,\ldots,f_6)$ in the curve
$$C_{\underline{f}}\colon Y^2=f(X)=f_6 X^6+ f_5 X^5+\cdots+f_0$$
over $\kf=k(f_6,f_5,\ldots,f_0)$.
Similarly, any curve of genus $2$ with all of its Weierstrass points labeled can be obtained by specializing $(w_1,\ldots,w_6,f_6)$ in the curve
$$C_{\underline{w}}\colon Y^2=f_6(X-w_1)\cdots(X-w_6)$$
over $\kw=k(f_6,w_1,\ldots,w_6)$. Of course, one can just forget a labelling to obtain a curve $C_{\underline{f}}$ from $C_{\underline{w}}$. 
This allows us to express $\kw$  as a finite extension of
$\kf$ via
\begin{align*}
f_5&=-f_6(w_1+\cdots+w_6)\\
f_4&=f_6(w_1w_2+w_1w_3+\cdots+w_5w_6)\\
&\vdots\\
f_0&=f_6w_1\cdots w_6
\end{align*}
In fact, $\kw$ is a splitting-field of $f(X)$ over $\kf$ and $\Gal(\kw/\kf)=S_6$. As we observed in the proof of Theorem~\ref{T:galchar}, $\kw$ is also the splitting field of $\Jac(C_{\underline{f}})[2]$ over $\kf$.
The fractional linear transformations on the $X$-line below $C$ induce a $\mathrm{PGL}_2(k)$-action on $\kf$ and $\kw$.
If we divide out by this action, we obtain a relation with the function fields of the coarse moduli spaces $\MM$ of curves of genus $2$ and $\MM(2)$ of curves of genus $2$ with full level $2$-structure on their Jacobians, which is an $\Sp_4(\FF_2)$-cover of $\MM$
$$\xymatrix@C3em@R1em{
\kw \ar[dr]^{./\mathrm{PGL}_2(k)}\ar[dd]_{./S_6}\\
&k(\MM(2))\ar[dd]^{./\mathrm{Sp}_4(\FF_2)}\\
\kf \ar[dr]_{./\mathrm{PGL}_2(k)}\\
&k(\MM)\\
}$$
The subgroups identified in Theorem~\ref{T:galchar} give rise to intermediate fields $K_1,K_2,K_3$ as depicted in Figure~\ref{F:galgroups} and, by dividing out by $\mathrm{PGL}_2(\FF_2)$, also moduli spaces between $\MM$ and $\MM(2)$. One of the interesting phenomena here, that does not occur for elliptic curves, is that there are two non-conjugate ways of specifying two maximal isotropic subgroups of $\Jac(C)[2]$ and hence that there are multiple partial level $2$ structures that can be imposed on $\Jac(C)[2]$.

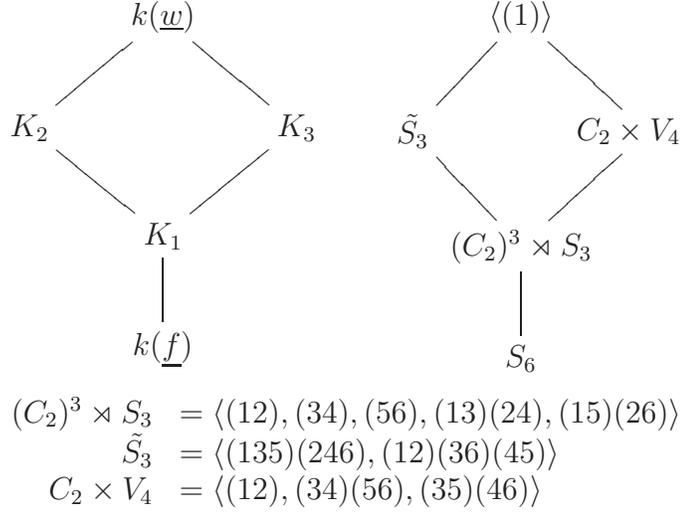
\begin{figure}
$$
\xymatrix{
&\kw\ar@{-}[dl]\ar@{-}[dr]\\
K_2\ar@{-}[dr]&&K_3\ar@{-}[dl]\\
&K_1\ar@{-}[d]\\
&\kf
}\hspace{2em}
\xymatrix@C0.1em{
&\langle(1)\rangle\ar@{-}[dl]\ar@{-}[dr]\\
\tilde{S_3}\ar@{-}[dr]&&\hspace{-1.2em}C_2\ar@{-}[dl]\times V_4\\
&(C_2)^3\rtimes S_3\ar@{-}[d]\\
&S_6
}$$
$$
\begin{array}[t]{rl}
 (C_2)^3\rtimes S_3&=\langle (12),(34),(56),(13)(24),(15)(26)\rangle\\
 \tilde{S_3}&=\langle (135)(246),(12)(36)(45)\rangle\\
C_2\times V_4&=\langle (12),(34)(56),(35)(46)\rangle
\end{array}
$$
\caption{Galois groups associated to intermediate $2$-level structure}\label{F:galgroups}
\end{figure}


\section{Bielliptic genus 2 curves with \texorpdfstring{$S_3$}{S3} as a Galois group}\label{S:S3}

In Section~\ref{S:4x4to2x2} we saw that a $(4,4)$-splitting $E_1\times E_2\to J$ gives rise to a $(2,2)$-splitting $E_1\times E_2\to A$, where $A$ is a principally polarized abelian surface admitting two rational polarized $(2,2)$-isogenies with trivially intersecting kernels.

In this section, we give something close to a universal model for the genus 2 curve $C_2$ from Lemma~\ref{T:C4andC2}. Since the corresponding moduli space of genus $2$ curves is not a fine moduli space (the space $\MM(2)$ is not even fine), a universal curve does not exist. However, by allowing extra parameters, we can still give a family that covers all possible $C_2$ by specialization, similar to how any elliptic curve can be obtained by specializing a general Weierstrass model $Y^2+a_1XY+a_3Y=X^3+a_2X^2+a_4X+a_6$.

Let $k$ be a field of characteristic distinct from 2 or 3.  Let $C_2$ be a genus 2 curve over $k$ with a $(2,2)$-split Jacobian and let $E_1$ be a degree 2 subcover of $C_2$.  Then $E_1$ has a model $V^2 = f(U) = U^3 + bU + c$ and $\Gal(f)$, the Galois group of $f$, is a subgroup of $S_3$.  
In order to produce the family, we concentrate on the most general case $\Gal(f)=S_3$. We will argue later that other cases are also parametrized.

By Theorem~\ref{T:2x2model}, the curve $C_2$ admits a model
 $Y^2 = g(X)$  where
\[
g(X) = f\left(\frac{X^2}{d} + a\right) \text{ with }a,d\in k.
\]

Working in the extension $k[U]/(f(U)) = k[r]$, the polynomials $f$ and $g$ factor as
\begin{equation}\label{E:sextic}
\begin{aligned}
f(U)  &= (U - r) \left(U^2 + rU + \left(r^2 + b\right)\right)\\
g(X)  &= \frac{1}{d^3}\left(X^2 + ad - rd\right)h(x),
\end{aligned}
\end{equation}	
where
\begin{equation}\label{E:h-quartic}
h(X) = X^4 + (dr + 2ad)X^2 + d^2\left(r^2 + ar + a^2 + b\right).
\end{equation} 

By Lemma~\ref{T:C4andC2}, we know that $g$ and $f$ have the same splitting field.
This means that $h$ must be reducible over $k(r)$.  Otherwise $h$ would be irreducible and we would require a degree 4 extension over $k(r)$ to split $h$.  The following lemma gives a testable condition.

\begin{lem}[Kappe and Warren \cite{MR992075}]\label{T:irreducibility}
Let $h(x) = x^4 + bx^2 + d$ be a polynomial over a field $k$ of characteristic $\neq 2$ and let $\pm \alpha$, $\pm \beta$ be its roots.  Then the following conditions are equivalent:
\begin{enumerate}
	\item $h(x)$ is irreducible over $k$;
	\item \label{L:temp} The following are not squares in $k$:
	\begin{enumerate}[(i)]
		\item $b^2 - 4d$, 
		\item $-b + 2\sqrt{d}$, and
		\item $-b -2\sqrt{d}$.
	\end{enumerate}
\end{enumerate}
\end{lem}

We can use Lemma \ref{T:irreducibility} to determine the conditions on $a$ and $d$ such that $h$ factors as a product of two quadratics over $k(r)$.  In our case, the polynomial $h$ will be reducible over $k(r)$ if one of the following is true:
\begin{enumerate}[\hspace{1.5em}(i)]
	\item \label{L:Case1}$(dr + 2ad)^2 - 4d^2\left(r^2 + ar + a^2 + b\right)$ is a square in $k(r)$, or
	\item \label{L:Case2}$-(dr + 2ad) + 2d\sqrt{r^2 + ar + a^2 + b}$ is a square in $k(r)$, or
	\item \label{L:Case3}$-(dr + 2ad) - 2d\sqrt{r^2 + ar + a^2 + b}$ is a square in $k(r)$.
\end{enumerate}

Taking the conditions one at a time, in case \eqref{L:Case1}, after simplification, we require $-3r^2 - 4b$ to be a square.  Observe that this is the discriminant of $x^2 + rx + (r^2 + b)$ and hence occurs exactly when our original polynomial $f(x)$ splits over $k(r)$. This contradicts $\Gal(f) = S_3$, so we ignore this possibility for now.

In the remaining two cases, we require $r^2 + ar + a^2 + b$ to be a square in $k(r)$.  Let $t \in k(r)$ such that $r^2 + ar + a^2 + b = t^2$.  Since $k(r)$ is a cubic extension of $k$, we can set $t = t_2r^2 + t_1r + t_0$. It follows that
\begin{align*}
	r^2 + ar + a^2 + b &= \left(t_2r^2 + t_1r + t_0\right)^2\\
	&= t_2^2r^4 + 2t_1t_2r^3 + \left(t_{1}^{2} + 2t_0t_2\right)r^2 + 2t_0t_1r + t_{0}^{2}\\
	&=\left(t_{1}^{2} + 2t_0t_2 - bt_{2}^{2}\right)r^2 + \left(2t_0t_1 - 2bt_1t_2 - ct_{2}^{2}\right)r + \left(t_{0}^{2} - 2ct_1t_2\right).
\end{align*}

Equating coefficients, we obtain the system of three equations:
\begin{align}
	t_{1}^{2} + 2t_0t_2 - bt_{2}^{2} - 1 &= 0\notag\\
	-a + 2t_0t_1 - 2bt_1t_2 - ct_{2}^{2} &=0\\
	a^2 + b -t_{0}^{2} + 2ct_1t_2 &=0\notag
\end{align}

We obtain an affine variety in $\mathbb{A}^{4}$ with parameters $b$ and $c$. 
This variety has two components, interchanged by
 $(a,t_0,t_1,t_2) \mapsto (a, -t_0, -t_1, -t_2)$, which can be found either using a primary decomposition of a polynomial ideal (e.g., \texttt{PrimaryComponents} in Magma \cite{magma}), or by eliminating variables (say $a$ and $t_0$) via resultants and multivariate GCD, and a multivariate polynomial factorization. Each component is a genus 0 curve in $\mathbb{A}^{4}$.  Using for instance Magma, we can parametrize this curve. Writing $s$ for the parameter, we obtain

\begin{align}
	a &= \frac{s^{4} - 2bs^{2} - 8cs + b^{2}}{4\left(s^{3} + bs + c\right)}\label{E:Param-a}\\
	t_{0} &= \frac{-s^4 - 6bs^{2}  - 4cs - b^{2}}{4\left(s^{3} + bs + c\right)}\notag    \\
	t_{1} &= \frac{-s^{3} + bs + 2c}{2\left(s^{3} + bs + c\right)}\notag    \\
	t_{2} &= \frac{-3s^{2} - b}{2\left(s^{3} + bs + c\right)}.\notag    
\end{align}

For any $s \in k$, this parametrization gives a value for $a$ such that \mbox{$r^2 + ar + a^2 + b$} 
is a square in $k(r)$.  Using the parametrization, we can express the square root of \mbox{$r^2 + ar + a^2 + b$} as
$$
\frac{-3s^{2} - b}{2\left(s^{3} + bs + c\right)}r^2 + 
	\frac{-s^{3} + bs + 2c}{2\left(s^{3} + bs + c\right)}r + 
	\frac{-s^4 - 6bs^{2}  - 4cs - b^{2}}{4\left(s^{3} + bs + c\right)}.
$$
This allows us to evaluate the expressions in \eqref{L:Case2} and \eqref{L:Case3}.
In case \eqref{L:Case2} we find that $-(dr + 2ad) + 2d\sqrt{r^2 + ar + a^2 + b}$ becomes
	$$\left(-\frac{1}{4\left(s^3 + bs + c\right)}\right) \cdot d \cdot F_1,$$
where $F_1 = \left(6s^3 + 2bs\right) r^2 - \left(6s^3 + 2bs\right) r - \left(3s^4 + 2bs^2 - 12cs + 3b^2\right)$. This is a square in $k(r)$ if and only if
\begin{equation}\label{E:dbad}
d = -\left(s^{3} + bs + c\right) \cdot \square
\end{equation}
where $\square$ represents a square in $k$.
Using \eqref{E:Param-a} and \eqref{E:dbad}, we find that $g(X) = f(X^2/d + a)$ has the same splitting field as $f$.  The Galois group of $g$ is indeed isomorphic to $S_3$ but its representation in $S_6$ is $S''_3 =\langle (123)(456),(23)(56)\rangle$ which is not conjugate to $\tilde{S_3}$ from Section \ref{S:2Torsion}. Therefore, $C\colon Y^2=g(X)$ is not of the form predicted by Lemma~\ref{T:C4andC2}.

In case \eqref{L:Case3}, we find that $-(dr + 2ad) - 2d\sqrt{r^2 + ar + a^2 + b}$ becomes
	$$
	\left(-\frac{1}{4\left(s^3 + bs + c\right)}\right) \cdot d \cdot F_2
	$$	where $F_2 = \left(6s^2 + 2b\right)r^2 - \left(2s^3 + 6bs + 8c\right) r - \left(s^4 + 10bs^2 - 20cs + b^2\right)$.  This is a square in $k(r)$ if and only if
\begin{equation}\label{E:Param-d}
\begin{aligned}
	d &= \p{4b^3 + 27c^2}\p{s^3 + bs + c} \cdot \square\\
	&= -D\cdot f(s) \cdot \square
\end{aligned}
\end{equation}
where $\square$ represents a square in $k$ and $D$ is the discriminant of $f$.

Using this parametrization, our hyperelliptic curve $C_2$ is given by $Y^2 = g(X)$ where:
\begin{equation}\label{E:C2}
\begin{aligned}
	g =&\frac{1}{{s^3 + bs + c}^3}\Bigg(\frac{1}{\p{4b^3 + 27c^2}^3}X^6 + \frac{3\p{s^4 - 2bs^2 - 8cs + b^2}}{4\p{4b^3 + 27c^2}^2}X^4\\
	&+ \frac{P(b,c,s)}{16\p{4b^3 + 27c^2}}X^2+ \frac{\p{s^6 + 5bs^4 + 20cs^3 - 5b^2s^2 - 4bcs - b^3 - 8c^2}^2}{64}\Bigg)
\end{aligned}
\end{equation}
and where $P$ is given by
$$
P = 3s^8 + 4bs^6 - 48cs^5 + 50b^2s^4 + 128bcs^3 + 4b^3s^2 + 192c^2s^2 - 16b^2cs + 3b^4 + 16bc^2.
$$
As desired, we find that $g$ has the same splitting field as $f$ and that $\Gal(g)\simeq \tilde{S_3}$ as found in Section \ref{S:2Torsion}.  
The factorization for $g$ over its splitting field is given in appendix \ref{S:factoredg}.  

Let $\phi_1\colon C_2\to E_1$ be the cover arising from $(X,Y)\mapsto(U,V)=(X^2/d+a,Y)$. Let $\Psi\colon\Jac(C_2)\to B$ be one of the other polarized $(2,2)$-isogenies we have by construction on $\Jac(C_2)$. 
Let 
\[E_s\colon W^2=-\disc(f)\cdot f(s)\cdot(U-a)\cdot f(U)\]
be the complementary curve and $\phi_2\colon C_2\to E_s$ the corresponding cover.
It is straightforward to check that $\Psi\circ \phi^*\colon E_1\to B$ is injective and hence that $\Phi_4=\Psi\circ(\phi_1^*+\phi_2^*)\colon E_1\times E_s\to B$ is an optimal $(4,4)$-splitting of $B$. This means that the data we have specified ($s$ and $\Psi$) should also determine an anti-isometry $\alpha_s\colon E_1[4]\to E_s[4]$. The ambiguity of choice in $\Psi$ corresponds to the fact that if $\alpha_s\colon E_1[4]\to E_s[4]$ is an anti-isometry, then so is $-\alpha_s$.

Let $X_{E_1}^{-}(4)$ be the completion of the moduli space of elliptic curves with prescribed $4$-torsion structure anti-isometric to $E_1[4]$ modulo multiplication by $(\ZZ/4\ZZ)^\times$. This is a cover of the $j$-line $X(1)$, Galois over $\kbar$, with
\[\Aut_{\kbar}(X_{E_1}^{-}(4)/X(1))=\PSL_2(\ZZ/4\ZZ),\]
 so $X_{E_1}^{-}(4)\to X(1)$ is a degree $24$ cover.

On the open part of the $s$-line where the equation for $E_s$ defines an elliptic curve, the map $s\mapsto E_s$ provides a map from the $s$-line to $X_{E_1}^{-}(4)$. 
This map cannot be constant, since $a(s)$ is not constant in $s$, so we can interpret the $s$-line as a cover of $X_{E_1}^{-}(4)$. It turns out to be an isomorphism, and $E_s$ provides a model of the universal elliptic curve over $X_{E_1}^{-}(4)$. This provides an alternative construction to the one given by Silverberg \cite{MR1638488}. Our formulas are shorter.

\begin{prop}\label{P:X4-Es} Let $b,c\in k$ such that $4b^3+27c^2\neq 0$ and let
\[E\colon V^2=f(U)=U^3+bU+c\]
be an elliptic curve. Let $s$ be a parameter on $\PP^1$ and consider
\[E_s\colon W^2=-\disc(f)\left(4(s^3+bs+c)U-(s^4 - 2bs^2 - 8cs + b^2)\right) f(U),\]
with $(U,W)=(\frac{s^4 - 2bs^2 - 8cs + b^2}{4(s^3+bs+c)},0)$ taken to be the identity element.
Then $E_s$ is isomorphic to
\begin{align*}
\tilde{E_s}\colon
y^2&=x^3+a_4x+a_6\text{ with}\\
a_4&=(4b^3 + 27c^2)^2(s^8b + 12s^7c - 28/3s^6b^2 - 28s^5bc - 14/3s^4b^3 -84s^4c^2 \\
   &+ 28/3s^3b^2c - 28/3s^2b^4 - 56s^2bc^2 - 44/3sb^3c - 96sc^3 + b^5 + 20/3b^2c^2)\\
a_6&=-(4b^3 + 27c^2)^3(s^{12}c - 8/3s^{11}b^2 - 22s^{10}bc + 88/27s^9b^3 - 88s^9c^2 +
55s^8b^2c\\
   & - 176/9s^7b^4 - 308/9s^6b^3c - 176/9s^5b^5 - 176s^5b^2c^2 - 649/9s^4b^4c - 528s^4bc^3\\
   & + 88/27s^3b^6 - 704/9s^3b^3c^2 - 704s^3c^4 + 154/9s^2b^5c + 352/3s^2b^2c^3 -
8/3sb^7\\
  & - 248/9sb^4c^2 - 64sbc^4 - 5/3b^6c - 560/27b^3c^3 - 64c^5)
\end{align*}
with
\[j(\tilde{E}_s)=\frac{256}{4b^3+27c^2}
 \frac{\left(
\begin{aligned}
3bs^8 + 36cs^7 - 28b^2s^6 - 84bcs^5 - 14(b^3 + 18c^2)s^4
+ 28b^2cs^3 \\- 28b(b^3 + 6c^2)s^2 - 4c(11b^3 + 72c^2)s +
3b^5 + 20b^2c^2
\end{aligned}
\right)^3
}{(s^6 + 5bs^4 + 20cs^3 - 5b^2s^2 - 4bcs - b^3 - 8c^2)^4}.
\]
The map $s\mapsto E_s$ induces an isomorphism $\PP^1\to X_{E}^{-}(4)$ and $\tilde{E}_s$ provides a model of the universal curve over $X_{E}^{-}(4)$.  
For $s=\infty$ we find that $\tilde{E}_\infty$ is isomorphic to the quadratic twist $E^{(D)}$ of $E$ by $D=\disc(E)$.
\end{prop}
\begin{proof}
The computation of the model $\tilde{E}_s$ and its $j$-invariant are straightforward. It establishes that $s\mapsto j(E_s)$ induces a degree $24$ cover $\PP^1\to X(1)$. We have already established that $s\mapsto E_s$ induces a cover $\PP^1\to X_{E_1}^{-}(4)$. The map induced by $s\mapsto j(E_s)$ factors through $j\colon X_{E_1}^{-}(4)\to X(1)$, which also has degree $24$, so the first map must be of degree $1$ and hence an isomorphism.

The only point where the curve defined by $\tilde{E}_s$ might not be immediately clear is for $s=\infty$. However, we can consider the isomorphic model $y^2=x^3+a_4/s^8x+a_6/s^{12}$. Then we find that
\[
\left.\frac{a_4}{s^8}\right|_{s=\infty}=(4b^3+27c^2)^2b\text{ and } \left.\frac{a_6}{s^{12}}\right|_{s=\infty}=-(4b^3+27c^2)^3c
\]
which confirms that $E_\infty=E^{(D)}$ with $D=-16(4b^3+27c^2)=\disc(E)$.
\end{proof}

Note that the description of $E_s$ is even shorter than that of $\tilde{E}_s$, but $E_s$ has the drawback of not being a Weierstrass-form and not specializing to an elliptic curve for $s^3+bs+c=0$. It does show very nicely where the denominator of $j(\tilde{E}_s)$ comes from. This denominator vanishes exactly when $f(\frac{s^4 - 2bs^2 - 8cs + b^2}{4(s^3+bs+c)})=0$.

\begin{cor}\label{C:sline-44split-modspace} Let $E_1\colon V^2=U^3+bU+c$ be an elliptic curve
The affine variety $\PP^1_s\setminus\{s^6 + 5bs^4 + 20cs^3 - 5b^2s^2 - 4bcs - b^3 - 8c^2\}$ parametrizes principally polarized abelian surfaces $J_s$ together with a pair of optimal $(4,4)$-splittings $\pm\Phi_4\colon E_1\times E_s \to J_s$.
\end{cor}

\begin{cor}\label{cor:4torsiontwist}
Let $E$ be an elliptic curve over a field $k$ with $\charact(k)\neq 2$. Let $D$ be the discriminant of $E$. Then there is an anti-isometry $\alpha_4\colon E[4]\to E^{(D)}[4]$.
\end{cor}

\begin{proof} Apart from the proof by specialization that is part of Proposition~\ref{P:X4-Es},
there is also a Galois-representation theoretic way of proving Corollary~\ref{cor:4torsiontwist}. This is interesting because it identifies how $\GL_2(\ZZ/4\ZZ)$ permits such an anti-isometry. Let $E$ be an elliptic curve over a field $k$ with discriminant $D$ and let $\rho\colon\Gal(\kbar/k)\to \Aut(E[4])$ be the mod $4$ Galois representation. We have $\Aut(E[4])\simeq \GL_2(\ZZ/4\ZZ)$. Let $H$ be the subgroup of elements that act via even permutation on the $2$-torsion elements. Note that $D$ is also the discriminant of the $2$-torsion algebra, so $\rho^{-1}(H)=\Gal(\kbar/k(\sqrt{D}))$.

Consider
$$M=\begin{pmatrix}1&2\\2&-1\end{pmatrix}\in\GL_2(\ZZ/4\ZZ)$$
and let $\alpha_M\colon E[4]\to E[4]$ be the corresponding automorphism.
One can check that $\{M,-M\}$ is the unique conjugacy class of $\GL_2(\ZZ/4\ZZ)$ of size $2$ and that the centralizer of $M$ is $H$.
It follows that $\alpha_M$ is defined over $k(\sqrt{D})$.
Furthermore, since
$$M\begin{pmatrix}0&1\\-1&0\end{pmatrix}M^T=\begin{pmatrix}0&-1\\1&0\end{pmatrix}$$
we see that $\alpha_M\colon E[4]\to E[4]$ is an \emph{anti}-isometry.

Now consider the quadratic twist $E^{(D)}$ of $E$. There is an isomorphism $E\to E^{(D)}$ defined over $k(\sqrt{D})$, which when restricted, yields an isometry $\alpha^{(D)}\colon E[4]\to E^{(D)}[4]$. The composition
$\alpha^{(D)}\circ \alpha_M\colon E[4]\to E^{(D)}[4]$ is an anti-isometry.
Furthermore, if $\sigma\in\Gal(\kbar/k)$ and $\sigma(\sqrt{D})=-\sqrt{D}$ then $\sigma(\alpha_M)=-\alpha_M$ and $\sigma(\alpha^{(D)})=-\alpha^{(D)}$. Hence, $\sigma(\alpha^{(D)}\circ\alpha_M)=\alpha^{(D)}\circ\alpha_M$, so we see that $E[4]$ and $E^{(D)}[4]$ are anti-isometric over $k$.
\end{proof}

In Sections~\ref{S:22split} and \ref{S:4x4to2x2}, we already observed that the abelian variety $A$ in \eqref{E:44split} is generally a Jacobian, a sufficient condition being that $j(E_1)\neq j(E_2)$. We now establish that the model $C_2\colon Y^2=g(X)$ with $g(X)$ as in \eqref{E:44split} specializes to a genus $2$ curve such that $A=\Jac(C_2)$ whenever $A$ is a Jacobian.

\begin{lem}\label{T:C2}
Let $\Phi_4\colon E_1\times E_2\to J$ be an optimal $(4,4)$-splitting and let $\Phi_2\colon E_1\times E_2\to A$ be the induced $(2,2)$-splitting as in \eqref{E:44split}. Suppose we have a model $E_1\colon V^2=U^3+bU+c$. If $A=\Jac(C_2)$, where $C_2$ is some genus $2$ curve, then for some $s\in k$ we obtain a model
\begin{equation}\label{eq:C2model}
C_2\colon Y^2=g(X) \text{ with $g(x)$ as in \eqref{E:C2}}.
\end{equation}
Conversely, any $C_2$ of this form admits a $(2,2)$-splitting $\Phi_2$ and a further polarized $(2,2)$-isogeny $\Psi\colon \Jac(C_2)\to J$ such that $\Psi\circ\Phi_2\colon E_1\times E_2\to J$ is an optimal $(4,4)$-splitting.
\end{lem}

\begin{proof}
An optimal $(4,4)$-splitting is specified by an anti-isometry $\alpha_4\colon E_1[4]\to E_2[4]$. It follows from Proposition~\ref{P:X4-Es} that $E_2\simeq E_s$ for some value of $s\in k\cup\{\infty\}$. If $s=\infty$ or $s^3+bs+c=0$, we have $j(E_1)=j(E_2)$ and the induced isometry $E_1[2]\to E_2[2]$ is the obvious one. In this case, $A\simeq E_1\times E_1$ or $A\simeq \Re_{k(\sqrt{d})/k}(E_1)$, see Section~\ref{S:22split}. For all the other cases, the discriminant of the polynomial $g(X)$ defined in \eqref{E:C2} is square-free as long as $j(E_s)\neq \infty$.

For the converse, $\Jac(C_2)$ admits an obvious $(2,2)$-splitting $\Phi_2\colon E_1\times E_2\to \Jac(C_2)$. Furthermore, there are two further $(2,2)$-isogenies $\Psi$ defined on $\Jac(C_2)$ by construction. Let $\phi_1\colon C_2\to E_1$ be the corresponding double cover. It is straightforward to check that $\phi_1^*(E_1[2])\cap \ker(\Psi)=0$ and hence that $\Psi\circ\Phi_2$ is an optimal $(4,4)$-splitting.
\end{proof}

\section{A model for genus \texorpdfstring{$2$}{2} curves with \texorpdfstring{$(4,4)$}{(4,4)}-split Jacobian}\label{S:C4derivation}

The next step is to describe a model for a genus $2$ curve $C_4$ with a $(4,4)$-split Jacobian. From Section~\ref{S:4x4to2x2} we know that $\Jac(C_4)$ is the image under a $(2,2)$-isogeny of a $(2,2)$-split principally polarized abelian surface $A$, admitting three $(2,2)$-isogenies with pairwise trivially intersecting kernels. Whenever $A=\Jac(C_2)$, then Lemma~\ref{T:C2} gives us a model for $C_2$. Section~\ref{S:22maps} provides an explicit description of $(2,2)$-isogenies between Jacobians of genus $2$ curves.

In this section, we will identify the $(2,2)$-isogenies of $\Jac(C_2)$ defined over $k$ and derive a description of the codomain, if it is a Jacobian. This provides us with a description of $C_4$ with $(4,4)$-split Jacobian in case $A=\Jac(C_2)$.

We consider all 15 different quadratic splittings as in Section~\ref{S:22maps} over $\kbar$ and see which are defined over the base field.
As expected, a computer calculation yields that  one of the quadratic splittings is singular. It is
$$
\set{q_2(X-w_1)(X-w_2), q_2(X-w_3)(X-w_4), q_2(X - w_5)(X-w_6)},
$$ 
where $w_i$ are the roots of $g$ over $k[r,R]$ as listed in Appendix~\ref{S:factoredg} and $q_2^3 = f_6$ is the leading coefficient of $g$.  This singular splitting is due to the $(2,2)$-isogeny $\Phi_2^*\colon \Jac(C_2) \longrightarrow E_1 \times E_2$.
We also find that applying the Richelot correspondence \eqref{eq:Cdtilde} to the 14 generically non-singular quadratic splittings produces only two $k$-rational sextics, with the remaining twelve defined over $\bar{k}$, but not over $k$.  The two quadratic splittings which yield the $k$-rational sextics are
\begin{align}
	&\set{q_2(X-w_1)(X-w_6),q_2(X-w_2)(X-w_3),q_2(X - w_4)(X-w_5)} \mbox{ and}\label{L:R1}\\
	&\set{q_2(X-w_1)(X-w_4),q_2(X-w_2)(X-w_5),q_2(X - w_3)(X-w_6)}.\label{L:R2}
\end{align} 

Notice that the singular quadratic splitting, together with the two quadratic splittings \eqref{L:R1} and \eqref{L:R2} come from the three partitionings that are fixed by $\tilde{S_3}$, given by \eqref{E:S3}.  

Let $G_1$ and $G_2$ denote the sextics obtained by applying Richelot's construction \eqref{eq:Cdtilde} of $f$ to the quadratic splittings \eqref{L:R1} and \eqref{L:R2} respectively.  We find that $G_2(X) = G_1(-X)$, and therefore that both models are isomorphic. This reflects that $E_1\times E_2$ has an extra automorphism $[1]\times[-1]$, so if $\Phi_4$ is an optimal $(4,4)$-splitting then $\Phi_4\circ ([1]\times[-1])$ is another one, with the the same codomain.

Proposition~\ref{T:twist} allows us to select the right twist
$$
C_4\colon Y^2 = DG_1(X) = F(X) \text{ where }D = \disc(f) = -4b^3 - 27c^2
$$
(see Appendix~\ref{S:C4} for $F(X)$, with the extraneous factor $f_6^2$ removed). Looking at the denominators and the discriminant of the sextic given in Appendix~\ref{S:C4}, we find
$$
\disc(F) = \frac{2^6\p{s^3 + bs + c}^{22}
		\p{s^6 + 5bs^4 + 20cs^3 - 5b^2s^2 - 4bcs - b^3 - 8c^2}}
	{\p{4b^3 + 27c^2}^{14}\p{3bs^4 + 18cs^3 - 6b^2s^2 - 6bcs - b^3 - 9c^2}^{18}}
$$
and hence
\begin{prop}\label{T:Conditions} 
The model $C_4:Y^2=F(X)$ with $F(X)$ as defined in Appendix~\ref{S:C4} describes a genus $2$ curve unless one of the following holds:
\begin{enumerate}
	\item \label{degen:disc} $4b^3+27c^2=0$,
	\item \label{degen:favanish} $s^6 + 5bs^4 + 20cs^3 - 5b^2s^2 - 4bcs - b^3 - 8c^2 = 0$,
	\item \label{degen:singquadsplit} $3bs^4 + 18cs^3 - 6b^2s^2 - 6bcs - b^3 - 9c^2=0$, or
        \item \label{degen:twotor} $s^3 + bs + c = 0$
	\item \label{degen:sinf} $s=\infty$
\end{enumerate}
The cases \eqref{degen:disc} and \eqref{degen:favanish} correspond to situations where either $E_1$ or $E_s$ is not an elliptic curve. The cases \eqref{degen:singquadsplit} and \eqref{degen:twotor} correspond to $(4,4)$-split principally abelian surfaces that are not Jacobians, as described by Propositions~\ref{P:n-1isog} and \ref{P:44degen} respectively.

If $j(E_1)\neq 0$ then the case~\eqref{degen:sinf} corresponds to a $(4,4)$-splitting
$\Phi\colon E_1\times E_1^{(D)}\to \Jac(C_4')$, where
\begin{equation}\label{ref:Csinfty}
C_4'\colon Y^2=-64bc\frac{1}{D^3}X^6+\frac{64}{3}b\frac{1}{D^2}X^5+16bc\frac{1}{D^2}X^4+\frac{224}{27}b\frac{1}{D}X^3+
4bc\frac{1}{D}X^2+\frac{4}{3}bX-bc,
\end{equation}
is a curve of genus $2$. If $j(E_1)=0$ then case \eqref{degen:sinf} is part of case \eqref{degen:singquadsplit}.
\end{prop}

\begin{proof}
\begin{enumerate}
 \item[\eqref{degen:disc}] In this case $E_1$ is not an elliptic curve.
 \item[\eqref{degen:favanish}] In Proposition~\ref{P:X4-Es} we have already seen that this relation implies $j(E_s)=\infty$.
 \item[\eqref{degen:singquadsplit}] Let $\delta$ denote the determinant of the quadratic splitting \eqref{L:R1}. Then
$$
N_{k[r, R]/k}(\delta) = \p{4b^3 + 27c^2}^2\p{3bs^4 + 18cs^3 - 6b^2s^2 - 6bcs -b^3 - 9c^2}^2,
$$
and we know that if \eqref{E:delta} vanishes, then the codomain of the $(2,2)$-isogeny is a product of elliptic curves over $\kbar$.
Proposition~\ref{P:X3relation} explains this degeneracy.
 \item[\eqref{degen:twotor}] If $s^3+bs+c=0$ then $(s,0)\in E_1[2]$ is a point of order two. Furthermore, from \eqref{E:Param-a} we have $a(s)=\infty$, so the $(2,2)$-splitting $\Phi_2$ through which our $(4,4)$-splitting factors is known to be $E_1\times E_1\to E_1\times E_1$ or $E_1\times E_1^{(D)}\to \Re_{k(\sqrt{R})/k}(E_1)$, depending on whether $D=\disc(E_1)$ is a square or not.
Since $\#\SL_2(\ZZ/4\ZZ)=8\cdot\#\SL_2(\ZZ/2\ZZ)$, there are $8$ ways over $\kbar$ to extend a $(2,2)$-isogeny to a $(4,4)$ isogeny. This also follows from the computation in the proof of Lemma~\ref{L:44iso-kernels}, where one finds $8$ possible kernels for $\Psi$ trivially intersecting $\ker(\Phi^*)$.
Since every value of $s$ gives rise to two $(4,4)$-isogenies, we see that with $s=\infty$ and $s^3+bs+c=0$, all possible $(4,4)$-splittings factoring through $\Phi_2$ must occur for these values of $s$.  Proposition~\ref{P:44degen} describes six $(4,4)$ splittings of this type, so together with the two coming from $s=\infty$, these must be all.
 \item[\eqref{degen:sinf}] The model for $C_4$ as presented does not specialize well for $s=\infty$, but the isomorphic model
$$C_{4,s}'\colon(s^3Y)^2=F(xs^2)/s^6$$
does if $b\neq0$ and for $s=\infty$ we obtain $C_4'$. Note that for generic $s$ we have a $(4,4)$-splitting
\[
\Phi_4\colon E_1\times E_s \to \Jac(C_{4,s}')
\]
where the kernel is the graph of an anti-isogeny $E_1[4]\to E_s[4]$ that is independent of $s$. Since domain and codomain specialize well at $s=\infty$, so must the $(4,4)$-isogeny.
If $b=0$ we have $j(E_1)=0$ and we have a $3$-isogeny $E_1\to E_1^{(D)}$, so case \eqref{degen:singquadsplit} applies.
\end{enumerate}
\end{proof}

\begin{prop}\label{P:X3relation}
Let $E$ and $E_s$ be the elliptic curves described in Proposition~\ref{P:X4-Es}. The
relation $3bs^4 + 18cs^3 - 6b^2s^2 - 6bcs -b^3 - 9c^2=0$ corresponds to the existence of a $3$-isogeny $\phi\colon E\to E_s$.
\end{prop}

\begin{proof}
Note that $J_1=j(E)$ and $J_2=j(E_s)$ are rational functions in $s,b,c$. We can express the given information in weighted homogeneous polynomial relations in $s,b,c$ with weights $(1,2,3)$ and obtain
\begin{align*}
1728b^3-(b^3 + 27c^2/4)J_1&=0\\
N(b,c,s)-D(b,c,s)J_2&=0\\
3bs^4 + 18cs^3 - 6b^2s^2 - 6bcs -b^3 - 9c^2&=0.
\end{align*}
When we eliminate $b,c,s$ from these equations, we are left with the classical modular polynomial of level $3$ in $J_1,J_2$. This means that there is a $3$-isogeny $\phi\colon E\to E_s$ over $\kbar$.
Indeed, in the light of Proposition~\ref{P:n-1isog} we expect to find $(4,4)$-split surfaces of this type that are not Jacobians.
A priori, the fact that the $j$-invariants satisfy a modular polynomial only tells us that $E$ and $E_s$ are $3$-isogenous over $\kbar$. However, we know that only one twist of $E_s$ has $E_s[4]$ anti-isometric to $E[4]$ and similarly, only one twist of $E_s$ can be $3$-isogenous to $E$. From Proposition~\ref{P:n-1isog} we know they must coincide.

Furthermore, in general there are only $2$ anti-isometries between $E[4]$ and $E_s[4]$ for $3$-isogenous curves (otherwise $E[4]$ would have extra automorphisms, requiring the Galois representation to be small). This implies that the anti-isometries induced by our parametrization of $X_{E}^{-}(4)$ must coincide with the ones from Proposition~\ref{P:n-1isog} generally and therefore also for any valid specialization of $b,c,s$.
\end{proof}


\section{Proofs of Theorems~\ref{T:44splitclassification} and \ref{T:main}}\label{S:Proof}

\begin{proof}[Proof of Theorem~\ref{T:44splitclassification}]
In Section~\ref{S:4x4to2x2} we established that a $(4,4)$-splitting factors $\Phi_4$ as
\[\xymatrix{
E_1\times E_2\ar[r]^-{\Phi_2}\ar@/_2ex/[rr]_-{\Phi_4} &A\ar[r]^\Psi& J,
}\]
Let $E_1\colon V^2=U^3+bU+C$ be a model for $E_1$.
If $A$ is a Jacobian then Lemma~\ref{T:C2} provides a model for $C_2$ such that $A=\Jac(C_2)$ and if $J$ is a Jacobian $\Jac(C_4)$ as well, then Section~\ref{S:C4derivation} shows that \eqref{E:C4} provides a model for $C_4$.

More generally, Corollary~\ref{C:sline-44split-modspace} describes that $b,c,s$ together parametrize all $J$ with optimal $(4,4)$-splitting. Proposition~\ref{T:Conditions} analyzes all the degeneracies of $C_4$ and identifies which correspond to the $(4,4)$-splittings described by Propositions~\ref{P:44degen} and \ref{P:X3relation}. Together, these give the cases listed in Theorem~\ref{T:44splitclassification}.
\end{proof}

\begin{proof}[Proof of Theorem~\ref{T:main}] Recall that the moduli space of genus $2$ curves is birational to $\AA^3$ and that $(i_1,i_2,i_3)$ as given in \eqref{E:AbsInv}, give coordinates on that space. 

Let $\mathcal{X}$ denote the surface inside $\AA^3$ describing genus 2 curves with $(4,4)$-split Jacobians.  This surface is the Humbert surface of discriminant 16 and it is irreducible (see \cite[Corollaries~1.6-1.8]{MR1285957} and \cite{MR1285952}).  

Theorem~\ref{T:44splitclassification} and Proposition~\ref{T:Conditions} show that by specializing $b,c,s$ we can generate points on a Zariski-open part of $\mathcal{X}$. In fact, if we set $b=1$, the points we can generate still lie dense in $\mathcal{X}$. By computing the Igusa invariants of $C_4$ we obtain rational functions $i_1(c,s),i_2(c,s),i_3(c,s)\in \QQ(c,s)$ such that the image of the rational map
\[
\begin{array}{ccc}
\AA^2&\dashrightarrow&\AA^3\\
(c,s)&\mapsto&(i_1(c,s),i_2(c,s),i_3(c,s))
\end{array}
\]
lies dense in $\mathcal{X}$. The defining equations are too large to compute the image using Gr{\"o}bner bases or resultants. Instead, we will compute the image by interpolation. Our strategy consists of three steps.
\begin{enumerate}
 \item Determine a candidate equation $\LL(i_1,i_2,i_3)=0$ to describe $\mathcal{X}$,
 \item prove that $\mathcal{X}$ is contained in $\LL(i_1,i_2,i_3)=0$,
 \item observe that if the Zariski-closure of $\mathcal{X}$ is a proper subset of $\LL(i_1,i_2,i_3)=0$ then $\mathcal{X}$ must lie on a surface of lower degree and derive a contradiction from that.
\end{enumerate}

For (1), we guessed degree bounds with which to interpolate $\LL$ and
computed a tentative version $\LL_{p_i}(i_1,i_2,i_3)\pmod{ p_i}$ via interpolation, for
93 consecutive 6-digit primes $p_i$.
For future reference, note that we found a unique solution to the system for each prime $p_i$.

We then used rational reconstruction to compute a tentative equation $\LL(i_1,i_2,i_3)=0$ over $\QQ$.
The equation of the surface is too large to reproduce here: $\LL$ contains 4574 monomials with coefficients of up to 138 digits. Note that the information we computed should allow us to construct $\LL \pmod {N}$, where $N = \prod_{i=1}^{93}p_i \approx 10^{600}$, so the coefficients we found in $\LL$ are relatively tiny. This is a strong indicator that we have computed something that indeed has intrinsic meaning over $\QQ$ (at this point, basically what could go wrong is that our degree bound is too low and that we very unluckily have picked interpolation points that happen to map to points satisfying some lower degree equation as well).

For (2), we show that $\LL(i_1(c,s),i_2(c,s),i_3(c,s))$ is  identically zero in $\QQ(c,s)$.
The expression $\LL\p{i_1(c,s), i_2(c,s), i_3(c,s)} = 0$ gives rise, after clearing denominators, to a polynomial $p(c,s)$ of degrees at most 1800 and 4050 in $c$ and $s$ respectively. We need to establish that $p(c,s)=0$ as a bivariate polynomial.
Expanding $p(c,s)$ explicitly is computationally infeasible, so instead, we evaluate $p(c,s)$ over a large number of distinct values for $c$ and $s$.  For a fixed value $s = s_0$, if we show that $p(c, s_0) = 0$ at 1801 distinct values for $c$, then $p(c, s_0)$ is the zero polynomial on the line $s = s_0$.  If we repeat this process on 4501 distinct lines $s = s_i$ then $p(c,s)$ is in fact the zero polynomial.  This calculation was performed in parallel on multiple computers over the course of several weeks.

For (3), note that we have now established that the Zariski-closure of $\mathcal{X}$ is indeed contained in $\LL(i_1,i_2,i_3)=0$. Since $\mathcal{X}$ is irreducible (see \cite[Corollary~1.8]{MR1285957}), proper containment implies that $\mathcal{X}$ must be described by an equation of strictly lower degree. But then we would have found this lower degree equation in step (1) as well. However, we found there that $\LL_{p_i}$ was the unique equation below the guessed degree bounds that interpolated the computed images. So $\mathcal{X}$ does not lie in a lower degree surface.
\end{proof}


\appendix
\section{On a classical result by Bolza}\label{S:Bolza}

An 1887 paper by O.\ Bolza \cite{Bolza1887} discusses hyperelliptic integrals which can reduce into elliptic integrals by a fourth degree transformation.  In the terminology of Section~\ref{sec:nnsplit}, he computes a model of a genus 2 curve with a $(4,4)$-split Jacobian. In this section we relate his results to ours. The formulas given here are available electronically from \cite{equation}.
Bolza works over $\CC$. 
He gives a $3$-parameter family of curves $y^2 = R(x)$, with parameters $\lambda,\mu,\nu$, with a sign error in the equation \eqref{Bolza:nu} below. Corrected, Bolza's family is given by:
\begin{align*}
	C_{(\lambda,\mu,\nu)}\colon y^2=R(x) = \nu'x^6 &- 6\lambda\nu'x^5 
		+ 3\p{4 \mu \nu' + \lambda \mu'}x^4 + 2\p{\lambda\lambda' + 5\nu\nu'}x^3\\
	&+ 3\p{4\mu'\nu + \lambda'\mu}x^2 - 6\lambda'\nu x + \nu,
\end{align*}
where
\begin{align}
	\lambda' &= -\frac{1}{3}\cdot\frac{2\lambda^2\nu - \lambda\mu^2 - \mu\nu}
			{-\nu^2 + 3\lambda\mu\nu - 2\mu^3},\label{Bolza:lambda}\\
	\mu' &= \phantom{-}\frac{1}{9}\cdot\frac{\lambda^2\mu + \lambda\nu - 2\mu^2}
			{-\nu^2 + 3\lambda\mu\nu - 2\mu^3},\label{Bolza:mu}\\
	\nu' &= -\frac{1}{27}\cdot\frac{2\lambda^3 - 3\lambda\mu + \nu}
			{-\nu^2 + 3\lambda\mu\nu - 2\mu^3}.\label{Bolza:nu}
\end{align}
He also gives the variable substitutions that turn the hyperelliptic integrals into elliptic integrals. In modern language, he gives the degree $4$ maps from the curve $C_{(\lambda,\mu,\nu)}$ to two elliptic curves. Since Bolza is only interested in curves over $\CC$, he does not care to determine the appropriate twist, but this is easily adjusted. With
$$z_1=\frac{\lambda x^4+4\lambda\nu x+3\mu\nu}{\lambda x^2+2\lambda x+\frac{3\mu\lambda-2\nu}{2}},
\quad
z_2=\frac{\lambda'+4\lambda'\nu' x^3+3\mu'\nu'x^4}{x^2(\lambda'+2\lambda' x+\frac{3\mu'\lambda'-2\nu'}{2}x^2)}$$
we find that $C_{(\lambda,\mu,\nu)}$ covers the two curves
\[
\begin{split}
E_{1,(\lambda,\mu,\nu)}\colon w_1^2=\lambda R_1(z_1)=\lambda&(\lambda z_1 -2\nu) (\nu'z_1^3-3(9\lambda^2\nu'-6\mu\nu'-\lambda\mu')z_1^2\\
&+12(9\lambda\nu\nu'+3\mu'\nu+\lambda'\mu)z_1+12\nu(3\mu\mu'-\lambda\lambda'))
\end{split}
\]
and
\[
\begin{split}
E_{2,(\lambda,\mu,\nu)}\colon w_2^2=\lambda' R_2(z_2)=\lambda'&(\lambda' z_2 -2\nu')
(\nu z_2^3-3(9\lambda'^2\nu-6\mu'\nu-\lambda'\mu)z_2^2\\
&+12(9\lambda'\nu'\nu+3\mu\nu'+\lambda\mu')z_2+12\nu'(3\mu'\mu-\lambda'\lambda))
\end{split}.
\]
Checking this is straightforward by verifying that $\lambda R_1(z_1)R(x)$ and $\lambda' R_2(z_2)R(x)$ are squares in $\QQ(\lambda,\mu,\nu)(x)$.

In order to find the relation between Bolza's family and the model \eqref{E:C4}, we put $E_{1,(\lambda,\mu,\nu)}$ in short Weierstrass form $V^2=U^3+bU+c$, where
\[
\begin{split}
b=3 (\nu^2 - 3 \nu \mu \lambda + 2 \mu^3)^2&(2 \nu^4 \mu - 5 \nu^4 \lambda^2 + 2 \nu^3 \mu \lambda^3 + 16 \nu^3 \lambda^5 - \nu^2 \mu^4 +\\
&10 \nu^2 \mu^3 \lambda^2 - 45 \nu^2 \mu^2 \lambda^4 - 6 \nu \mu^5 \lambda + 36 \nu \mu^4 \lambda^3 - 9 \mu^6 \lambda^2)\\
c=(\nu^2 - 3 \nu \mu \lambda + 2 \mu^3)^3&(\nu^7 - 3 \nu^6 \mu \lambda - 10 \nu^6 \lambda^3 - 10 \nu^5 \mu^3 + 84 \nu^5 \mu^2 \lambda^2 -138 \nu^5 \mu \lambda^4 + 160 \nu^5 \lambda^6 -\\
&30 \nu^4 \mu^4 \lambda + 68 \nu^4 \mu^3 \lambda^3 -
78 \nu^4 \mu^2 \lambda^5 - 288 \nu^4 \mu \lambda^7 - 2 \nu^3 \mu^6 + 30 \nu^3 \mu^5 \lambda^2 -\\
&189 \nu^3 \mu^4 \lambda^4 + 738 \nu^3 \mu^3 \lambda^6 - 18 \nu^2 \mu^7 \lambda + 198 \nu^2 \mu^6 \lambda^3 -
729 \nu^2 \mu^5 \lambda^5 -\\
&54 \nu \mu^8 \lambda^2 + 324 \nu \mu^7 \lambda^4 - 54 \mu^9 \lambda^3).
\end{split}
\]
We compute the linear transformation $U=\frac{t_{1}z_2+t_{2}}{t_3z_2+t_4}$ such that $\lambda'R_2(z_2)=d(U-a)(U^3+bU+c)$, where $d$ is specified up to squares, and find
\[
\begin{split}
 a&=\frac{(\nu^2 - 3 \nu \mu \lambda + 2 \mu^3) (2 \nu^3 \lambda - 3 \nu^2 \mu^2 - 4 \nu^2 \lambda^4 + 2 \nu \mu^3 \lambda + 6 \nu \mu^2 \lambda^3 - 3 \mu^4 \lambda^2)}{\mu \lambda-\nu}\\
d&=3 (\nu - \mu \lambda) (\nu^2 - 3 \nu \mu \lambda + 2 \mu^3) (\nu^2 - 6 \nu \mu \lambda + 4 \nu \lambda^3 + 4 \mu^3 - 3 \mu^2 \lambda^2).
\end{split}
\]
From $a=\frac{s^4-2bs^2-8cs+b^4}{4(s^3+bs+c)}$ one finds one rational choice:
$$s=\frac{(\nu^2 - 3 \nu \mu \lambda + 2 \mu^3) (\nu^3 \lambda + 3 \nu^2 \mu^2 - 18 \nu^2 \mu \lambda^2 + 16 \nu^2 \lambda^4 + 10 \nu \mu^3 \lambda -
15 \nu \mu^2 \lambda^3 + 3 \mu^4 \lambda^2)}{\nu-\mu \lambda}.$$
This shows that outside $(\nu - \mu \lambda) (\nu^2 - 3 \nu \mu \lambda + 2 \mu^3)=0$, Bolza's family
maps to the family \eqref{E:C4}. The relation turns out to be birational: both $(\lambda:\mu:\nu)$ and $(s:b:c)$ are naturally coordinates on weighted projective space $\PP(1,2,3)$. The formulae above express $(b/s^2,c/s^3)$ as functions in $(\mu/\lambda^2,\nu/\lambda^3)$. Via the appropriate resultant computations and polynomial factorizations, we find
\[
\begin{split}
\psi(b,c,s)&=2 b^6 + 36 b^5 s^2 + 45 b^4 c s + 72 b^4 s^4 + 45 b^3 c^2 + 36 b^3 c s^3 - 36 b^3 s^6 + 297 b^2 c^2 s^2 - 378 b^2 c s^5\\
& + 54 b^2 s^8 + 324 b c^3 s - 81 b c^2 s^4 + 324 b c s^7 + 216 c^4 - 324 c^3 s^3 + 891 c^2 s^6 - 27 c s^9\\
\frac{\mu}{\lambda^2}&=\frac{(2 b^4 - 15 b^2 c s + 30 b^2 s^4 + 9 b c^2 + 90 b c s^3 + 135 c^2 s^2 -
27 c s^5)\psi(b,c,s)}
{3 (b s + c + s^3)^2 (b^2 - 6 b s^2 - 12 c s - 3 s^4)^2 (4 b^3+27 c^2)}\\
\frac{\nu}{\lambda^3}&=\frac{-\psi(b,c,s)^2}
{(b s + c + s^3)^2 (b^2 - 6 b s^2 - 12 c s - 3 s^4)^3 (4 b^3+27 c^2)}
\end{split}
\]
This shows that outside some codimension one locus, the two families parametrize the same curves up to twist. Note, however, that the formulas for $a,b,c,d$ are of weighted total degrees $13,26,39,15$ in $(\lambda,\mu,\nu)$.
That means that with appropriate scaling, we can adjust the square class of $d$, so the two families really do parametrize essentially the same curves.

\section{The six roots of the defining polynomial for \texorpdfstring{$C_2$}{C2}}\label{S:factoredg}

Let $C_2$ be a genus 2 curve over $k$ which is $(2, 2)$-isogenous to a genus 2 curve whose Jacobian is optimally $(4, 4)$-split (see Lemma \ref{T:C4andC2}).  Then $C_2$ is a degree 2 cover of an elliptic curve $E_1$ which admits a model $V^2 = f(U) = U^3 + bU + c$.  A model for $C_2$ is given in \eqref{E:C2}.
$$
f(U)  = (U - r) \left(U^2 + rU + \left(r^2 + b\right)\right)
$$
Over $k[r,R]:=k[r][U]/\left[U^2 - \left(-3r^2 - 4b\right)\right]$, we have the factorisation
$$
f(U) = (U - r)\p{U - \frac{R}{2} + \frac{r}{2}}\p{U + \frac{R}{2} + \frac{r}{2}}.
$$
Using our parametrization for $a$ and $d$ given in equations \eqref{E:Param-a} and \eqref{E:Param-d} respectively, we can write down the factorization for $g$ over $k[r,R]$: 
\begin{equation}\label{E:factoredg}
g(X) = f_6\prod_{i=1}^{6} (X - w_i)
\end{equation}
where $$f_6 =  \p{\frac{1}{-\disc(f)\cdot f(s)}}^3 = \frac{1}{\p{4b^3  + 27c^2}^3\p{s^3 + bs + c}^3}$$
and:
{\allowdisplaybreaks
\begin{align*}
	\begin{split}
	w_1 = {}\frac{1}{2}&\Big(\p{-3s^2 - b}r^2 + \p{-4bs - 6c}r - bs^2 - 6cs + b^2\Big)R\\
	\end{split}\\
	\begin{split}
	w_2 = {}\frac{1}{2}&\Big(\p{3s^2 + b}r^2 + \p{4bs + 6c}r + bs^2 + 6cs - b^2\Big)R\\
	\end{split}\\
	\begin{split}
	w_3 = {}\frac{1}{2}&\Big(\p{-3s^2 - b}r^2 + \p{2bs + 3c}r - bs^2 + 3cs - b^2\Big)R\\
		&+ \frac{1}{2}\Big(\p{-3bs - 9c}r^2 + \p{9cs - 2b^2}r - 4b^2s - 6bc\Big)\\
	\end{split}\\
	\begin{split}
	w_4 = {}\frac{1}{2}&\Big(\p{3s^2 + b}r^2 + \p{-2bs - 3c}r + bs^2 - 3cs + b^2\Big)R\\
		&+ \frac{1}{2}\Big(\p{3bs + 9c}r^2 + \p{-9cs + 2b^2}r + 4b^2s + 6bc\Big)\\
	\end{split}\\	
	\begin{split}
	w_5 = {}\frac{1}{2}&\Big(\p{-3s^2 - b}r^2 + \p{2bs + 3c}r - bs^2 + 3cs - b^2\Big)R\\		&+ \frac{1}{2}\Big(\p{3bs + 9c}r^2 + \p{-9cs + 2b^2}r + 4b^2s + 6bc\Big)\\
	\end{split}\\
	\begin{split}
	w_6 = {}\frac{1}{2}&\Big(\p{3s^2 + b}r^2 + \p{-2bs - 3c}r + bs^2 - 3cs + b^2\Big)R\\
		&+ \frac{1}{2}\Big(\p{-3bs - 9c}r^2 + \p{9cs - 2b^2}r - 4b^2s - 6bc\Big)\\
	\end{split}
\end{align*}}

\section{A representation for a \texorpdfstring{$(4,4)$}{(4,4)}-split genus 2 curve}\label{S:C4}

Let $E_1$ be an elliptic curve over $k$ given by $V^2 = U^3 + bU + c$ for scalars $b$ and $c$ and let $C_4$ be a genus 2 curve which is a degree 4 cover of $E_1$.  Then there exists a scalar $s$ such that a representation for $C_4$ is given by $Y^2 = F(X)$ where:

\begin{equation}\label{E:C4}
\begin{aligned}
	F(X) = {} &\frac{\p{s^3 + bs + c}\p{27cs^3 - 18b^2s^2 - 27bcs - 2b^3 - 27c^2}}
		{\p{4b^3 + 27c^2}^3\p{3bs^4 + 18cs^3 - 6b^2s^2 - 6bcs - b^3 - 9c^2}^3}X^6\\
		& + \frac{3\p{s^3 + bs + c}^2\p{3s^2 + b}}
		{\p{4b^3 + 27c^2}^2\p{3bs^4 + 18cs^3 - 6b^2s^2 - 6bcs - b^3 - 9c^2}^3}X^5\\
		& + \frac{3\p{s^3 + bs + c}E(b,c,s)}
		{4\p{4b^3 + 27c^2}^2\p{3bs^4 + 18cs^3 - 6b^2s^2 - 6bcs - b^3 - 9c^2}^3}X^4\\
		& + \frac{-\p{s^3 + bs + c}^2G(b,c,s)}
		{2\p{4b^3 + 27c^2}\p{3bs^4 + 18cs^3 - 6b^2s^2 - 6bcs - b^3 - 9c^2}^3}X^3\\
		& + \frac{-\p{s^3 + bs + c}H(b,c,s)}
		{16\p{4b^3 + 27c^2}\p{3bs^4 + 18cs^3 - 6b^2s^2 - 6bcs - b^3 - 9c^2}^3}X^2\\
		& + \frac{3\p{s^3 + bs + c}^2\p{3s^4 + 6bs^2 + 12 cs - b^2}J(b,c,s)}
		{16\p{3bs^4 + 18cs^3 - 6b^2s^2 - 6bcs - b^3 - 9c^2}^3}X\\
		& + \frac{-\p{s^3 + bs + c}J(b,c,s)K(b,c,s)}
		{64\p{3bs^4 + 18cs^3 - 6b^2s^2 - 6bcs - b^3 - 9c^2}^3}
\end{aligned}
\end{equation}
and where

\begin{align*}
	\begin{split}
		E(b,c,s) = {}&9cs^7 - 26b^2s^6 - 171bcs^5 + 34b^3s^4 - 333c^2s^4 + 155b^2cs^3 - 6b^4s^2
		 + 126bc^2s^2\\ &+ 7b^3cs + 144c^3s - 2b^5 - 17b^2c^2
	\end{split}\\
	G(b,c,s) = {} &7s^6 + 23bs^4 + 68cs^3 - 11b^2s^2 - 4bcs - 3b^3 - 20c^2\\
	\begin{split}
	H(b,c,s) = {} &27cs^{11} + 6b^2s^{10} + 585bcs^9 - 402b^3s^8 + 2349c^2s^8 - 3330b^2cs^7 + 
        460b^4s^6\\
        & - 6156bc^2s^6 + 1410b^3cs^5 - 7776c^3s^5 + 140b^5s^4 + 
        4230b^2c^2s^4 + 23b^4cs^3\\
        & + 3024bc^3s^3 + 46b^6s^2 + 516b^3c^2s^2 + 
        3024c^4s^2 + 5b^5cs - 48b^2c^3s + 6b^7\\ &+ 85b^4c^2 + 288bc^4
    \end{split}\\
    J(b,c,s) = {} &s^6 + 5bs^4 + 20cs^3 - 5b^2s^2 - 4bcs - b^3 - 8c^2\\
	\begin{split}
	K(b,c,s) = {}&27cs^9 - 54b^2s^8 - 324bcs^7 + 36b^3s^6 - 891c^2s^6 + 378b^2cs^5 - 72b^4s^4\\
		&+ 81bc^2s^4 - 36b^3cs^3 + 324c^3s^3 - 36b^5s^2 - 297b^2c^2s^2 - 45b^4cs - 
        324bc^3s\\ &- 2b^6 - 45b^3c^2 - 216c^4
	\end{split}
\end{align*}

\bibliographystyle{amsplain}

\providecommand{\bysame}{\leavevmode\hbox to3em{\hrulefill}\thinspace}
\renewcommand{\MR}{\relax\ifhmode\unskip\space\fi}
\providecommand{\MRhref}[2]{%
  \href{http://www.ams.org/mathscinet-getitem?mr=#1}{#2}
}
\providecommand{\href}[2]{#2}

\providecommand{\bysame}{\leavevmode\hbox to3em{\hrulefill}\thinspace}
\providecommand{\MR}{\relax\ifhmode\unskip\space\fi MR }
\providecommand{\MRhref}[2]{%
  \href{http://www.ams.org/mathscinet-getitem?mr=#1}{#2}
}
\providecommand{\href}[2]{#2}

\end{document}